\documentclass[leqno,11pt]{amsart}
\usepackage{amssymb}
\usepackage{amsfonts}
\usepackage{graphicx}
\usepackage{amscd}
\usepackage{amsmath}
\usepackage{amssymb}
\usepackage{graphicx,pstricks,pst-plot}
\usepackage[utf8]{inputenc}

\theoremstyle{plain}
\newtheorem{theorem}{\indent\sc Theorem}[section]
\newtheorem{lemma}[theorem]{\indent\sc Lemma}
\newtheorem{corollary}[theorem]{\indent\sc Corollary}
\newtheorem{proposition}[theorem]{\indent\sc Proposition}

\theoremstyle{definition}

\newtheorem{remark}[theorem]{\indent\sc Remark}
\newtheorem{example}[theorem]{\indent\sc Example}

\usepackage{hyperref, color}

\hypersetup{
    colorlinks=true,       
    linkcolor=blue,        
    citecolor=blue,        
    urlcolor=blue          
}

\usepackage[
  lmargin=2.5cm,
  rmargin=2.5cm,
  tmargin=3.0cm,
  bmargin=3.0cm
]{geometry}

\makeatletter
\@namedef{subjclassname}{%
  \textup{2024} Mathematics Subject Classification}
\makeatother

\numberwithin{equation}{section}

\title[On the Geometry of Complete Spacelike LW-Submanifolds in $\mathbb{L}_{q}^{n+p}$]{On the Geometry of Complete Spacelike LW-Submanifolds in Locally Symmetric Semi-Riemannian Spaces}

\author[Ara\'{u}jo, J.G.S. And Barboza, W.F.C]{Jogli G. S. Ara\'{u}jo$^1$, Weiller F. C. Barboza$^{*,2}$}

\address{$^1$ Jogli G. da Silva Araújo \endgraf Departamento de Matem\'atica \endgraf
Universidade Federal Rural de Pernambuco \endgraf
52.171-900 Recife, Pernambuco, Brazil}

\email{jogli.silva@ufrpe.br \endgraf}

\address{$^2$ Weiller F. Chaves Barboza \endgraf Unidade Acadêmica de Matem\'atica \endgraf
Universidade Federal de Campina Grande \endgraf
58.429-970 Campina Grande, Para\'{\i}ba, Brazil}

\email{weiller@mat.ufcg.edu.br \endgraf}



\thanks{$^\ast$ Corresponding author.}

\begin{document}

\begin{abstract}
Let $M^{n}$ be an $n$-dimensional complete spacelike linear Weingarten submanifold immersed in a locally symmetric semi-Riemannian space $\mathbb{L}_{q}^{n+p}$ of index $q$, with parallel normalized mean curvature vector field and flat normal bundle. Assuming that $M^{n}$ satisfies suitable curvature constraints, we investigate rigidity results for such submanifolds. By combining a Simons-type formula for spacelike submanifolds with analytic techniques involving the Cheng--Yau modified operator $\mathcal{L}$, we establish sharp inequalities relating the traceless second fundamental form and the gradient of the mean curvature. As applications, we obtain several characterization results showing that $M^{n}$ must be either totally umbilical or isoparametric. More precisely, we derive rigidity results under three distinct frameworks: via the Omori--Yau maximum principle, via the $\mathcal{L}$-parabolicity of the underlying manifold, and under an integrability condition on the gradient of the mean curvature. These results generalize and unify known classification theorems for spacelike submanifolds satisfying linear Weingarten relations in semi-Riemannian ambient spaces.
\end{abstract}

\maketitle

\tableofcontents

\section{Introduction}

The study of spacelike submanifolds immersed in a semi-Riemannian space constitutes an important thematic from both physical and mathematical points of view. For instance, it was pointed out by J. Marsden and F. Tipler in \cite{Marsden:78} and S. Stumbles in \cite{Stumbles:80} that spacelike hypersurfaces with constant mean curvature in an arbitrary Lorentzian space (which is a semi-Riemannian space of index $p = 1$) play an important role in general relativity, in that they serve as convenient initial data for the Cauchy problem for Einstein equations. Furthermore, submanifold theory provides the adequate tools to approach some important problems involving spacetime singularities and gravitational collapse. The singularity theorems proved in the 1960s by R. Penrose in \cite{Penrose:1965} and S.W. Hawking and G. F. R. Ellis in \cite{Hawking-Ellis:1973} state that the formation of singularities is unavoidable, if one assumes reasonable conditions on the curvature of the spacetime, on the extrinsic geometry of certain submanifolds, and on the causal structure of the Lorentzian manifold. The existence of spacelike submanifolds in the spacetime, in particular, is a key requirement in the original formulation of the singularity theorems as well as in their more recent generalizations. For more details, see G. J. Galloway and J. M. M. Senovilla in \cite{Galloway-Senovilla:2010}, Z. Liang and X. Zhang in \cite{Liang-Zhang:2012}, and J. M. M. Senovilla in \cite{Senovilla:2012}.

From the mathematical point of view, the interest in the study of the geometry of these submanifolds is mostly due to the fact that they exhibit nice Bernstein$-$type properties, and one can truly say that the first remarkable results in this branch were the rigidity theorems of E. Calabi in \cite{Calabi: 70} and S. Y. Cheng and S. T. Yau in \cite{Cheng: 76}, who showed (the former for $n \leq 4$, and the latter for general $n$) that the only complete maximal spacelike hypersurfaces of the Lorentz Minkowski space $L^{n+1}$ are the spacelike hyperplanes. However, in the case that the mean curvature is a positive constant, A. E. Treibergs in \cite{Treibergs:82} astonishingly showed that there are many entire solutions of the corresponding constant mean curvature equation in $L^{n+1}$, which he was able to classify by their projective boundary values at infinity. When the ambient space is the de Sitter space $\mathbb S_{1}^{n+1}$, Cheng~\cite{Cheng: 90}, studying the case $b=0$, proved that if $M^{n}$ is a complete linear Weingarten spacelike hypersurface with nonnegative sectional curvature such that $H$ attains its maximum, then $M^{n}$ must be totally umbilical. on the other hand, S. Nishikawa in \cite{Nishikawa: 84} extended the results of \cite{Calabi: 70} and \cite{Cheng: 76} for a wide class of semi-Riemannian spaces, the so-called \emph{locally symmetric semi-Riemannian spaces}. At this point, it is worth recalling that a fundamental property of curvature is its control over the relative behavior of nearby geodesics. Because a normal neighborhood $\mathcal{U}$ is filled with radial geodesics, curvature thereby gives a description of the geometry of $\mathcal{U}$. Considering only the locally symmetric case, we have that this description is so accurate that, if $\mathcal{U}$ and $\tilde{\mathcal{U}}$ are normal neighborhoods with the same description (and same dimension and index), then $\mathcal{U}$ and $\tilde{\mathcal{U}}$ must be isometric. For more details concerning locally symmetric spaces, see Chapter 8 in \cite{O'Neill:83}. Returning to our context, we note that the seminal paper \cite{Nishikawa: 84} induced the appearance of several works approaching the problem of characterizing complete spacelike hypersurfaces immersed in a locally symmetric Lorentzian space, see, for instance, \cite{Baek: 04}, \cite{Lima: 13}, \cite{Lima: 13B}, \cite{Lima: 16B}, \cite{deLima:2017}, \cite{Liu: 10}. 

For higher codimension, considering again the particular case $b=0$, Liu~\cite{Liu:01} showed that the totally umbilical round spheres are the only $n$-dimensional compact linear Weingarten spacelike submanifolds of $\mathbb{L}_{p}^{n+p}$ with nonnegative sectional curvature and flat normal bundle. Generalizing the ideas of a previous work~\cite{Dan: 10}, Yang and Hou~\cite{Dan: 12} showed that a linear Weingarten spacelike submanifold in $\mathbb S_{p}^{n+p}$, with $a>0$, $b<1$, having parallel normalized mean curvature vector field (that is, the mean curvature function is positive and that the corresponding normalized mean curvature vector field is parallel as a section of the normal bundle) and such that the squared norm of its second fundamental form satisfies a suitable boundedness, must be either totally umbilical or isometric to a certain hyperbolic cylinder. On the other hand, still with higher codimensions, T. Ishihara in \cite{Ishihara: 88} applied a technique developed by S. S. Chern, M. P. do Carmo, and S. Kobayashi in \cite{Chern: 70} in order to extend the results of \cite{Calabi: 70} and \cite{Cheng: 76} to complete maximal spacelike submanifolds in a semi-Riemannian space form of constant nonnegative sectional curvature. More recently, De Lima et al.\ in~\cite{De Lima: 17} investigated complete maximal spacelike submanifolds immersed with flat normal bundle in a locally symmetric semi-Riemannian space obeying curvature conditions similar to those of S.~Nishikawa in~\cite{Nishikawa: 84}. In this setting, they obtained a suitable Simons type formula and, as application, they showed that such a spacelike submanifold must be totally geodesic or the square norm of its second fundamental form must be bounded, extending the results of T.~Ishihara in~\cite{Ishihara: 88} and S.~Nishikawa in~\cite{Nishikawa: 84}. Afterwards, J.~G.~Araujo et al.\ in~\cite{Araujo:17} studied $n$-dimensional complete linear Weingarten spacelike submanifolds $M^n$ with flat normal bundle and parallel normalized mean curvature vector field immersed in an $(n + p)$-dimensional locally symmetric semi-Riemannian manifold $\mathbb{L}_{p}^{n+p}$. 

In the work of De Lima et al. in~\cite{Araujo:17} they worked with space $\mathbb{L}^{n+p}_p$ and obtained sufficient conditions to guarantee that, in fact, $p = 1$ and $M^n$ is isometric to an isoparametric hypersurface of $\mathbb{L}^{n+1}_1$ having two distinct principal curvatures, one of which is simple. Next, working in this same context, the authors of the present paper, jointly with J.~G.~Araujo et al.\ in~\cite{Weiller:2019}, obtained another characterization result assuming an appropriate boundedness on the square norm of the second fundamental form of $M^n$ and considering the case that the ambient space $\mathbb{L}^{n+p}_p$ is also conformally flat, in order to reduce the codimension to $p = 1$. One year later, Barboza et al. in~\cite{Weiller:2023} studied complete linear Weingarten spacelike submanifolds with parallel normalized mean curvature vector field and flat normal bundle in a locally symmetric semi-Riemannian space $\mathbb{L}^{n+p}_p$ with index $p > 1$ and obeying the same set of curvature conditions assumed in~\cite{Weiller:2019} and~\cite{Araujo:17}. They established a more refined version of the Omori Yau maximum principle (see Proposition~1 in \cite{Weiller:2023}), which enables us to prove that such a spacelike submanifold must be either totally umbilical or isometric to an isoparametric submanifold of the ambient space (see Theorem~1 in \cite{Weiller:2023} ). Afterwards, they took over the parabolicity with respect to a modified Cheng Yau operator and an integrability property in order to get additional characterization results (see Theorems~2 and~3 in \cite{Weiller:2023}).

In this work, let us denote by $\mathbb{L}^{n+p}_q$ an $(n + p)$-dimensional locally symmetric semi-Riemannian space with index $q$ where $1 \leq q \leq p$. We recall that a semi-Riemannian space $\mathbb{L}^{n+p}_q$ with $1\leq q \leq p$ is said to be \emph{locally symmetric} when its curvature tensor $\overline{R}$ is parallel, that is, $\overline{\nabla} \overline{R} = 0$, where $\overline{\nabla}$ denotes the Levi Civita connection of $\mathbb{L}^{n+p}_q$. An $n$-dimensional submanifold $M^n$ immersed in $\mathbb{L}^{n+p}_q$ is said to be \emph{spacelike} if the induced metric on $M^n$ is positive definite.

We will denote by \text{LW-Submanifold} a submanifold $M^n$ of $\mathbb{L}^{n+p}_q$, $(1\leq q \leq p)$ \emph{linear Weingarten}, that is, a submanifold that its mean curvature $H$ and its normalized scalar curvature $R$ satisfy a linear relation of the type 
\begin{equation}\label{lw_relation}
R = aH + b, \quad \quad ~ a,b \in \mathbb{R}.
\end{equation}

Motivated by the rigidity phenomena for spacelike submanifolds satisfying Weingarten-type relations, in this paper we investigate complete spacelike linear Weingarten submanifolds (LW-submanifolds) immersed in the locally symmetric semi-Riemannian space $\mathbb{L}_{q}^{n+p}$, endowed with parallel normalized mean curvature vector field and flat normal bundle. Our approach combines analytic techniques involving the Cheng--Yau modified operator $\mathcal{L}$ with geometric inequalities derived from Simons-type formulas. In particular, we establish a sharp lower bound for $\mathcal{L}(nH)$ in terms of the traceless second fundamental form, which plays a central role throughout the paper.

More precisely, under suitable curvature assumptions on the ambient space and natural hypotheses on the second fundamental form, we obtain characterization and gap results showing that such submanifolds are either totally umbilical or isoparametric. These conclusions are achieved via three complementary methods: the Omori--Yau maximum principle, the $\mathcal{L}$-parabolicity of the underlying manifold, and an integrability condition on the gradient of the mean curvature. Each of these approaches leads to a rigidity phenomenon, revealing that the extremal behavior of the polynomial $P_{n,p,c,H}$ governs the geometry of the immersion.

The paper is organized as follows. In Section~\ref{sec:preliminaries} contains the necessary preliminaries on spacelike submanifolds in semi-Riemannian manifolds, together with the notation used throughout the paper. In Section~\ref{sec:simons}, we derive a Simons-type formula for spacelike submanifolds immersed in $\mathbb{L}_{q}^{n+p}$, which constitutes a fundamental analytic tool in our arguments. In Section~\ref{sec:key lemmas} is devoted to curvature constraints, explicit examples, and the key lemmas that will be employed in the proof of the main results. In Section~\ref{sec:main result}, we establish the characterization results for linear Weingarten spacelike submanifolds in $\mathbb{L}_{q}^{n+p}$. More precisely, Subsection~5.1 deals with rigidity results obtained via the Omori--Yau maximum principle, Subsection~5.2 addresses the characterization through the $\mathcal{L}$-parabolicity of the underlying manifold, and Subsection~5.3 is devoted to results obtained under an integrability condition on the gradient of the mean curvature.

\section{Preliminaries}\label{sec:preliminaries}

Let $M^n$ be an spacelike submanifold immersed in a locally symmetric semi-Riemannian space $\mathbb{L}_{q}^{n+p}$ of index $q$ with $1\leq q\leq p.$
We choose a local field of semi-Riemannian orthonormal frame $\{e_1,\cdots,e_{n+p}\}$ in $\mathbb{L}_{q}^{n+p},$
with dual coframe $\{\omega_1,\cdots,\omega_{n+p}\}$, such that, at each
point of $M^n$, $e_1,\cdots,e_n$ are tangent to $M^n$ and $e_{n+1},\cdots,e_{n+p}$ are normal to $M^n$. We have
that the pseudo-Riemannian metric $ds^2$ of $\mathbb{L}_{q}^{n+p}$ can be written as
\begin{eqnarray*}
d\overline{s}^2=\sum_{A}\epsilon_{A}\omega_{A}^2,
\end{eqnarray*}
where
$\epsilon_i=1, ~~ 1\leq i\leq n; \quad \epsilon_{\beta} = 1, ~~ n+1\leq \beta\leq n+p-q;\quad \epsilon_{\gamma}=-1, ~~ n+p-q+1\leq\gamma\leq n+p.$
\\
We will use the following convention for indices
$$1\leq A, B, C,\cdots\leq n + p; \quad 1\leq i, j, k,\cdots \leq n; \quad n + 1\leq \alpha, \alpha' \leq n + p.$$
$$n + 1\leq \beta,\beta'\leq n + p-q; \quad  n + p-q+1\leq \gamma,\gamma'\leq n+p.$$
Denoting by $\{\omega_{AB}\}$ the connection forms of $\mathbb{L}_{q}^{n+p}$, we have that the structure equations
of $\mathbb{L}_{q}^{n+p}$ are given by
\begin{eqnarray}
d\omega_{A}&=&\sum_{B}\epsilon_{B}\omega_{AB}\times \omega_{B},\quad \omega_{AB}+\omega_{BA}=0,
\end{eqnarray}
\begin{eqnarray}\label{derivada}
d\omega_{AB}&=&\sum_{C}\epsilon_{C}\omega_{AC}\times \omega_{CB}-\dfrac{1}{2}\sum_{C,D}\epsilon_{C}\epsilon_{D}\overline{R}_{ABCD}\omega_{C}\times \omega_{D},
\end{eqnarray}
where, $\overline{R}_{ABCD}$,$\overline{R}_{CD}$ and $\overline{R}$ denote respectively the Riemannian curvature
tensor, the Ricci tensor and the scalar curvature of the Lorentz space $\mathbb{L}_{q}^{n+p}.$
In this setting, we have
\begin{eqnarray*}
\overline{R}_{CD}&=&\sum_{B}\epsilon_{B}\overline{R}_{CBDB},\\
\overline{R}&=&\sum_{A}\epsilon_{A}\overline{R}_{AA}.
\end{eqnarray*}
We restrict forms to $M^n$, so that we have
$\omega_{\alpha}=0,\quad \alpha=n+1,\cdots, n+p,$
and the induced metric $ds^2$ of $M^n$ is written as $ds^2=\sum_{i}\omega_i^2.$
Since $\sum_{i}\omega_{\alpha i}\times \omega_{i}=d\omega_{\alpha}$ and by Cartan's Lemma we can write 
\begin{eqnarray}\label{COMUTATIVIDADE}
\omega_{i\alpha}=\sum_{j}h_{ij}^{\alpha}\omega_{j},\quad h_{ij}^{\alpha}=h_{ji}^{\alpha}.
\end{eqnarray}
The quadratic form
\begin{eqnarray*}
B=\sum_{i,j,\alpha}\epsilon_{\alpha}h_{ij}^{\alpha}\omega_i\omega_je_{\alpha},
\end{eqnarray*}
is the second fundamental form of $M^n$. Denote 
\begin{eqnarray*}
H^{\alpha}=\dfrac{1}{n}\sum_{i}h_{ii}^{\alpha},\quad \alpha=n+1,\cdots,n+p.
\end{eqnarray*}
The mean curvature vector $h$ is expressed as $h=\sum_{\alpha}\epsilon_{\alpha}H^{\alpha}e_{\alpha}$ and denoting by $H$ the length of $h$ and by $S$ the squared norm of the second fundamental form $B$, we have

\begin{eqnarray*}
H=||h||=\sum_{\alpha}(H^{\alpha})^2
\end{eqnarray*}
and
\begin{eqnarray*}
S=\sum_{\alpha,i,j}(h_{ij}^{\alpha})^2.
\end{eqnarray*}
We can write the structure equations of $M^{n}$ as follows
\begin{eqnarray*}
d\omega_i=\sum_{j}\omega_{ij}\times \omega_j,\quad \omega_{ji}+\omega_{ij}=0,\\
d\omega_{ij}=\sum_{k}\omega_{ik}\times \omega_{kj}-\dfrac{1}{2}\sum_{k,l}R_{ijkl}\omega_k\times \omega_l,
\end{eqnarray*}
where $R_{ijkl}$ are the components of the curvature tensor of $M^n.$ Using the previous structure
equations, we obtain the Gauss Equation
\begin{eqnarray}\label{EQUACAO DE GAUSS}
R_{mkjk}=\overline{R}_{mkjk}+\sum_{\beta}(h_{mj}^{\beta}h_{kk}^{\beta}-h_{mk}^{\beta}h_{kj}^{\beta})-\sum_{\gamma}(h_{mj}^{\gamma}h_{kk}^{\gamma}-h_{mk}^{\gamma}h_{kj}^{\gamma}).
\end{eqnarray}
Moreover, we get the Codazzi Equation
\begin{eqnarray}\label{EQUACAO DE CODAZZI}
h_{ijkk}^{\alpha}=\overline{R}_{\alpha ijkk}+h_{kijk}^{\alpha}.
\end{eqnarray}
In particular, the components of the Ricci tensor $R_{ik}$ and the normalized scalar curvature $R$
are given, respectively, by

\begin{eqnarray}
n(n-1)R=\sum_{i,j}\overline{R}_{ijij}+n^2\sum_{\beta}(H^{\beta})^2-n^2\sum_{\gamma}(H^{\gamma})^2-S_1+S_2,
\end{eqnarray}
where 
$$H^{\beta}=\dfrac{1}{n}\sum_{j}h_{jj}^{\beta},\quad H^{\gamma}=\dfrac{1}{n}\sum_{j}h_{jj}^{\gamma},\quad S_{1}=\sum_{\beta,i,j}(h_{ij}^{\beta})^2, \quad S_2=\sum_{\gamma,i,j}(h_{ij}^{\gamma})^2.$$
In this setting, we have $S = S_1+S_2.$ We will also consider that the second fundamental form is locally timelike, in the sense that, $H^\beta=0$ for $n+1\leq \beta \leq n+p-q$, thus
\begin{equation}\label{equacaoRS}
S_1 + n(n-1)R = \sum_{i,j} \overline{R}_{ijij} - n^2H^2 + S_2.
\end{equation}
Moreover, we suppose that $M^n$ has flat normal bundle, that is, $\overline{R}^{\perp}=0$ (equivalently $\overline{R}_{\alpha \beta jk}=0$), then $\overline{R}_{\alpha \beta jk}$ satisfy Ricci Equation

\begin{eqnarray}\label{EQUACAO DE RICCI}
\overline{R}_{\alpha \alpha' ij}=\sum_{k}(h_{ik}^{\alpha}h_{kj}^{\alpha'}-h_{kj}^{\alpha}h_{ik}^{\alpha'}).
\end{eqnarray}
Define the first and the second covariant derivarives of $\{h_{ij}^{\alpha}\}$, say $\{h_{ijk}^{\alpha}\}$ and $\{h_{ijkl}^{\alpha}\}$ by
$$
\begin{aligned}
\sum_{k} h_{ijk}^{\alpha}\omega_k 
  &= dh_{ij}^{\alpha} 
   + \sum_{k} \big(h_{kj}^{\alpha}\omega_{ki} + h_{ik}^{\alpha}\omega_{\beta \alpha}\big) 
   + \sum_{\beta} h_{ij}^{\beta}\omega_{\beta \alpha} 
   - \sum_{\gamma} h_{ij}^{\gamma}\omega_{\gamma \alpha}, \\[6pt]
\sum_{l} h_{ijkl}^{\alpha}\omega_l 
  &= dh_{ijk}^{\alpha} 
   + \sum_{l} h_{ijk}^{\alpha}\omega_{li} 
   + \sum_{l} h_{ilk}^{\alpha}\omega_{ij} 
   + \sum_{l} h_{ijl}^{\alpha}\omega_{lk} 
   + \sum_{\beta} h_{ijk}^{\beta}\omega_{\beta \alpha} 
   - \sum_{\gamma} h_{ijk}^{\gamma}\omega_{\gamma \alpha}.
\end{aligned}
$$
Furthermore, by exterior differentiation of~\eqref{derivada}, we get the following Ricci identity
\begin{eqnarray}\label{ricci identity}
h_{ijkl}^{\alpha}-h_{ijlk}^{\alpha}=\sum_{m}h_{mj}^{\alpha}R_{mikl}+\sum_{m}h_{mi}^{\alpha}R_{mjkl}.
\end{eqnarray}

\section{A Simons type formula for spacelike submanifolds in $\mathbb{L}^{n+p}_ q$}\label{sec:simons}

In what follows, we denote by $\nabla$ and $\Delta$ the gradient and Laplacian operator in the metric of the spacelike submanifold $M^n$. Then, the Laplacian of the second fundamental form $h^{\alpha}_{ij}$ is defined by 
\begin{equation}\label{definicao_Lap}
\Delta h^{\alpha}_{ij} = \sum_{k=1}^n h^{\alpha}_{ijkk}.
\end{equation}
From \eqref{EQUACAO DE CODAZZI} and \eqref{ricci identity}, we obtain
\begin{eqnarray}\label{laplac}
\Delta h_{ij}^{\alpha} &=& \sum_k h_{kkij}^{\alpha}+ \sum_{m,k}h_{mk}^{\alpha}R_{mijk}+\sum_{m,k}h_{im}^{\alpha}R_{mkjk}
\end{eqnarray}
$$
+\sum_{k,\beta}h_{ik}^{\beta}R_{\beta\alpha jk} -\sum_{k,\gamma}h_{ik}^{\gamma}R_{\gamma\alpha jk}.
$$
On the other hand, we have
\begin{eqnarray*}
\dfrac{1}{2}\Delta S=\sum_{i,j,\alpha}h_{ij}^{\alpha}\Delta h_{ij}^{\alpha}+\sum_{i,j,k,\alpha}(h_{ijk}^{\alpha})^2.
\end{eqnarray*}
Using the definition \eqref{definicao_Lap} and the fact that $|\nabla B|^2=\displaystyle\sum_{i,j,k}(h_{ijk}^{\alpha})^2$ we have

\begin{eqnarray*}
\dfrac{1}{2}\Delta S=\sum_{i,j,k,\alpha}h_{ij}^{\alpha}h_{ijkk}^{\alpha}+|\nabla B|^2.
\end{eqnarray*}
From the Codazzi Equation in ~\eqref{EQUACAO DE CODAZZI}
and the relaction ~\eqref{COMUTATIVIDADE} we get

\begin{eqnarray*}
\dfrac{1}{2}\Delta S&=&\sum_{i,j,k,\alpha}h_{ij}^{\alpha}\left(\overline{R}_{\alpha ijkk}+h_{kijk}^{\alpha}\right)+|\nabla B|^2\\
&=&\sum_{i,j,k,\alpha}h_{ij}^{\alpha}\overline{R}_{\alpha ijkk}+\sum_{i,j,k,\alpha}h_{ij}^{\alpha}h_{kijk}^{\alpha}+|\nabla B|^2.
\end{eqnarray*}
Thus, from of the ~\eqref{EQUACAO DE RICCI} and ~\eqref{ricci identity}, we have
\begin{eqnarray}\label{laplaciano de S}
\dfrac{1}{2}\Delta S&=&\sum_{i,j,k,\alpha}h_{ij}^{\alpha}\overline{R}_{\alpha ijkk}+\sum_{i,j,k,\alpha}h_{ij}^{\alpha}\left(h_{kikj}^{\alpha}+\sum_{k}h_{km}^{\alpha}R_{mijk}+\sum_{m}h_{mi}^{\alpha}R_{mkjk}\right)+|\nabla B|^2\nonumber \\
&=&|\nabla B|^2+\sum_{i,j,k,\alpha}h_{ij}^{\alpha}\overline{R}_{\alpha ijkk}+\sum_{i,j,k,\alpha}h_{ij}^{\alpha}h_{kikj}^{\alpha} +\sum_{i,j,k,m,\alpha}h_{ij}^{\alpha}h_{km}^{\alpha}R_{mijk} \nonumber \\ &&+\sum_{i,j,k,m,\alpha}h_{ij}^{\alpha}h_{mi}^{\alpha}R_{mkjk}.
\end{eqnarray}
From Codazzi Equation in ~\eqref{EQUACAO DE CODAZZI} and the fact that $(h_{ikl})_j=h_{iklj}$, we have
\begin{eqnarray}\label{equacao de codazzi 1}
h_{kikj}^{\alpha}=h_{kkij}^{\alpha}+\overline{R}_{\alpha kikj}.
\end{eqnarray}
Using ~\eqref{equacao de codazzi 1} in ~\eqref{laplaciano de S}, we get

\begin{eqnarray*}
\dfrac{1}{2}\Delta S&=&|\nabla B|^2+\sum_{i,j,k,\alpha}h_{ij}^{\alpha}\overline{R}_{\alpha ijkk}+\sum_{i,j,k,\alpha}h_{ij}^{\alpha}\left(h_{kkij}^{\alpha}+\overline{R}_{\alpha kikj}\right)\\
&&+\sum_{i,j,k,m,\alpha}h_{ij}^{\alpha}h_{km}^{\alpha}R_{mijk}+\sum_{i,j,k,m,\alpha}h_{ij}^{\alpha}h_{mi}^{\alpha}R_{mkjk} \\
&=&|\nabla B|^2+\sum_{i,j,k,\alpha}h_{ij}^{\alpha}\overline{R}_{\alpha ijkk}+\sum_{i,j,k,\alpha}h_{ij}^{\alpha}h_{kkij}^{\alpha}+\sum_{i,j,k,\alpha}h_{ij}^{\alpha}\overline{R}_{\alpha kikj} \\
&&+\sum_{i,j,k,m,\alpha}h_{ij}^{\alpha}h_{km}^{\alpha}R_{mijk}+\sum_{i,j,k,m,\alpha}h_{ij}^{\alpha}h_{mi}^{\alpha}R_{mkjk}.
\end{eqnarray*}
Thence,

\begin{eqnarray*}
\sum_{i,j,k,m,\alpha}h_{ij}^{\alpha}h_{km}^{\alpha}R_{mijk}&=&\sum_{i,j,k,m,\alpha}h_{ij}^{\alpha}h_{km}^{\alpha}\left[\overline{R}_{mijk}+\sum_{\beta}(h_{mj}^{\beta}h_{ik}^{\beta}-h_{mk}^{\beta}h_{ij}^{\beta}-\sum_{\gamma}(h_{mj}^{\gamma}h_{ik}^{\gamma}-h_{mk}^{\gamma}h_{ij}^{\gamma})\right]\\
&=&\sum_{i,j,k,m,\alpha}h_{ij}^{\alpha}h_{km}^{\alpha}\overline{R}_{mijk}+\sum_{i,j,k,m,\alpha,\beta}h_{ij}^{\alpha}h_{km}^{\alpha}h_{mj}^{\beta}h_{ik}^{\beta}-\sum_{i,j,k,m,\alpha,\beta}h_{ij}^{\alpha}h_{km}^{\alpha}h_{mk}^{\beta}h_{ij}^{\beta}\\
&&- \sum_{i,j,k,m,\alpha,\gamma}h_{ij}^{\alpha}h_{km}^{\alpha}h_{mj}^{\gamma}h_{ik}^{\gamma}+\sum_{i,j,k,m,\alpha,\gamma}h_{ij}^{\alpha}h_{km}^{\alpha}h_{mk}^{\gamma}h_{ij}^{\gamma},
\end{eqnarray*}
and 
\begin{eqnarray*}
\sum_{i,j,k,m,\alpha}h_{ij}^{\alpha}h_{mi}^{\alpha}R_{mkjk}&=&\sum_{i,j,k,m,\alpha}h_{ij}^{\alpha}h_{mi}^{\alpha}\left[\overline{R}_{mkjk}+\sum_{\beta}(h_{mj}^{\beta}h_{kk}^{\beta}-h_{mk}^{\beta}h_{kj}^{\beta})-\sum_{\gamma}(h_{mj}^{\gamma}h_{kk}^{\gamma}-h_{mk}^{\gamma}h_{kj}^{\gamma})\right]\\
&=&\sum_{i,j,k,m,\alpha}h_{ij}^{\alpha}h_{mi}^{\alpha}\overline{R}_{mkjk}+\sum_{i,j,k,m,\alpha,\beta}h_{ij}^{\alpha}h_{mi}^{\alpha}h_{mj}^{\beta}h_{kk}^{\beta} -\sum_{i,j,k,m,\alpha,beta}h_{ij}^{\alpha}h_{mi}^{\alpha}h_{mk}^{\beta}h_{kj}^{\beta} \\&&-\sum_{i,j,k,m,\alpha,\gamma}h_{ij}^{\alpha}h_{mi}^{\alpha}h_{mj}^{\gamma}h_{kk}^{\gamma} +\sum_{i,j,k,m,\alpha,\gamma}h_{ij}^{\alpha}h_{mi}^{\alpha}h_{mk}^{\gamma}h_{kj}^{\gamma}.
\end{eqnarray*}

At this point, we will deal with spacelike submanifolds $M^{n}$ of $L_{q}^{n+p}$ having parallel normalized mean curvature vector field, which means that the mean curvature function $H$ is positive and that the corresponding normalized mean curvature vector field $\frac{h}{H}$ is parallel as a section of the normal bundle. In this setting, we can choose a local orthonormal frame $\{e_{1},\ldots,e_{n+p}\}$ such that $e_{n+p-q+1}=\frac{h}{H}$, we have that

\begin{eqnarray}
H^{n+p-q+1}=\dfrac{1}{n}tr(h^{n+p-q+1})=H,
\end{eqnarray}
and 
\begin{eqnarray}
H^{\alpha}=\dfrac{1}{n}tr(h^{\alpha})=0,\quad \alpha > n+p-q+1.
\end{eqnarray}
Thus,
\begin{eqnarray}
\sum_{i,j,k,m,\alpha}h_{ij}^{\alpha}h_{mi}^{\alpha}R_{mkjk}&=&\sum_{i,j,k,m,\alpha}h_{ij}^{\alpha}h_{mi}^{\alpha}\overline{R}_{mkjk}+n\sum_{i,j,m,\alpha,\beta}h_{ij}^{\alpha}h_{mi}^{\alpha}h_{mj}^{\beta}H^{\beta} \nonumber \\
&&-\sum_{i,j,k,m,\alpha,\beta}h_{ij}^{\alpha}h_{mi}^{\alpha}h_{mk}^{\beta}h_{kj}^{\beta}-n\sum_{i,j,m,\alpha,\gamma}h_{ij}^{\alpha}h_{mi}^{\alpha}h_{mj}^{\gamma}H^{\gamma}\\
&&+\sum_{i,j,k,m,\alpha,\gamma}h_{ij}^{\alpha}h_{mi}^{\alpha}h_{mk}^{\gamma}h_{kj}^{\gamma}. \nonumber
\end{eqnarray}
Note that $n\displaystyle\sum_{i,j,m,\alpha,\beta}h_{ij}^{\alpha}h_{mi}^{\alpha}h_{mj}^{\beta}H^{\beta}=0$ because $H^{\beta}=0.$ Thus,
\begin{eqnarray}
\sum_{i,j,k,m,\alpha}h_{ij}^{\alpha}h_{mi}^{\alpha}R_{mkjk}&=&\sum_{i,j,k,m,\alpha}h_{ij}^{\alpha}h_{mi}^{\alpha}\overline{R}_{mkjk}-n\sum_{i,j,m,\alpha}h_{ij}^{\alpha}h_{mi}^{\alpha}h_{mj}^{n+p-q+1}H \nonumber \\
&& -\sum_{i,j,k,m,\alpha,\beta}h_{ij}^{\alpha}h_{mi}^{\alpha}h_{mk}^{\beta}h_{kj}^{\beta} \sum_{i,j,k,m,\alpha,\gamma}h_{ij}^{\alpha}h_{mi}^{\alpha}h_{mk}^{\gamma}h_{kj}^{\gamma}.
\end{eqnarray}
Therefore, 

\begin{eqnarray*}
\sum_{i,j,k,m,\alpha}h_{ij}^{\alpha}h_{km}^{\alpha}R_{mijk}+\sum_{i,j,k,m,\alpha}h_{ij}^{\alpha}h_{mi}^{\alpha}R_{mkjk}
&=&\sum_{i,j,k,m,\alpha}h_{ij}^{\alpha}h_{km}^{\alpha}\overline{R}_{mijk}+\sum_{i,j,k,m,\alpha,\beta}h_{ij}^{\alpha}h_{km}^{\alpha}h_{mj}^{\beta}h_{ik}^{\beta} \\&-&\sum_{i,j,k,m,\alpha,\beta}h_{ij}^{\alpha}h_{km}^{\alpha}h_{mk}^{\beta}h_{ij}^{\beta}-\sum_{i,j,k,m,\alpha,\gamma}h_{ij}^{\alpha}h_{km}^{\alpha}h_{mj}^{\gamma}h_{ik}^{\gamma}\\&+&\sum_{i,j,k,m,\alpha,\gamma}h_{ij}^{\alpha}h_{km}^{\alpha}h_{mk}^{\gamma}h_{ij}^{\gamma} + \sum_{i,j,k,m,\alpha}h_{ij}^{\alpha}h_{mi}^{\alpha}\overline{R}_{mkjk} \\
&-&n\sum_{i,j,m,\alpha}h_{ij}^{\alpha}h_{mi}^{\alpha}h_{mj}^{n+p-q+1}H - \sum_{i,j,k,m,\alpha,\beta}h_{ij}^{\alpha}h_{mi}^{\alpha}h_{mk}^{\beta}h_{kj}^{\beta}\\
&+&\sum_{i,j,k,m,\alpha,\gamma}h_{ij}^{\alpha}h_{mi}^{\alpha}h_{mk}^{\gamma}h_{kj}^{\gamma}.
\end{eqnarray*}

Observe that 
\begin{eqnarray*}
\sum_{i,j,k,\alpha}h_{ij}^{\alpha}\overline{R}_{\alpha ijkk}&=&\sum_{i,j,k,\alpha}h_{ij}^{\alpha}\overline{R}_{\alpha ijk;k}+\sum_{i,j,k,\alpha,\alpha'}h_{ij}^{\alpha}h_{ik}^{\alpha'}\overline{R}_{\alpha \alpha' jk} +
\sum_{i,j,k,\alpha,\alpha'}h_{ij}^{\alpha}h_{jk}^{\alpha'}\overline{R}_{\alpha i\alpha' k}\\&&+\sum_{i,j,k,\alpha,\alpha'}h_{ij}^{\alpha}h_{kk}^{\alpha'}\overline{R}_{\alpha ij \alpha'} +\sum_{i,j,k,m,\alpha}h_{ij}^{\alpha}h_{km}^{\alpha}\overline{R}_{mijk},
\end{eqnarray*}
using the Ricci Equation, we have 

\begin{eqnarray*}
\sum_{i,j,k,\alpha,\alpha'}h_{ij}^{\alpha}h_{ik}^{\alpha'}\overline{R}_{\alpha \alpha' jk}&=&\sum_{i,j,k,\alpha,\alpha'}h_{ij}^{\alpha}h_{ik}^{\alpha'}\left[\sum_{m}(h_{jm}^{\alpha}h_{mk}^{\alpha'}-h_{mk}^{\alpha}h_{jm}^{\alpha'})\right]\\
&=&\sum_{i,j,k,m,\alpha,\alpha'}h_{ij}^{\alpha}h_{ik}^{\alpha'}h_{jm}^{\alpha}h_{mk}^{\alpha'}-\sum_{i,j,k,m\,\alpha,\alpha'}h_{ij}^{\alpha}h_{ik}^{\alpha'}h_{mk}^{\alpha}h_{jm}^{\alpha'},
\end{eqnarray*}
thus, we conclude that,

\begin{eqnarray*}
\sum_{i,j,k,\alpha}h_{ij}^{\alpha}\overline{R}_{\alpha ijkk}&=&\sum_{i,j,k,\alpha}h_{ij}^{\alpha}\overline{R}_{\alpha ijk;k}+\sum_{i,j,k,m,\alpha,\alpha'}h_{ij}^{\alpha}h_{ik}^{\alpha'}h_{jm}^{\alpha}h_{mk}^{\alpha'} - \sum_{i,j,k,m,\alpha,\alpha'}h_{ij}^{\alpha}h_{ik}^{\alpha'}h_{mk}^{\alpha}h_{jm}^{\alpha'}\\&&+\sum_{i,j,k,\alpha,\alpha'}h_{ij}^{\alpha}h_{jk}^{\alpha'}\overline{R}_{\alpha i\alpha' k}+\sum_{i,j,k,\alpha,\alpha'}h_{ij}^{\alpha}h_{kk}^{\alpha'}\overline{R}_{\alpha ij \alpha'} + \sum_{i,j,k,m,\alpha}h_{ij}^{\alpha}h_{km}^{\alpha}\overline{R}_{mijk}.
\end{eqnarray*}
On the other hand,

\begin{eqnarray*}
\sum_{i,j,k,\alpha}h_{ij}^{\alpha}\overline{R}_{\alpha kikj}&=&\sum_{i,j,k,\alpha}h_{ij}^{\alpha}\overline{R}_{\alpha kik;j}+\sum_{i,j,k,\alpha,\alpha'}h_{ij}^{\alpha}h_{kj}^{\alpha'}\overline{R}_{\alpha \alpha' ik}+\sum_{i,j,k,\alpha,\alpha'}h_{ij}^{\alpha}h_{ij}^{\alpha'}\overline{R}_{\alpha k \alpha' k}\\
&&+\sum_{i,j,k,\alpha,\alpha'}h_{ij}^{\alpha}h_{kj}^{\alpha'}\overline{R}_{\alpha ki \alpha'}+\sum_{i,j,k,m,\alpha}h_{ij}^{\alpha}h_{jm}^{\alpha}\overline{R}_{mkik},
\end{eqnarray*}
hence,

\begin{eqnarray*}
\sum_{i,j,k,\alpha,\alpha'}h_{ij}^{\alpha}h_{kj}^{\alpha'}\overline{R}_{\alpha \alpha' ik}=\sum_{i,j,k,m,\alpha,\alpha'}h_{ij}^{\alpha}h_{kj}^{\alpha'}h_{im}^{\alpha}h_{mk}^{\alpha'}-\sum_{i,j,k,m,\alpha,\alpha'}h_{ij}^{\alpha}h_{kj}^{\alpha'}h_{mk}^{\alpha}h_{im}^{\alpha'},
\end{eqnarray*}
and yet, 

\begin{eqnarray*}
\sum_{i,j,k,\alpha}h_{ij}^{\alpha}\overline{R}_{\alpha kikj}&=&\sum_{i,j,k,\alpha}h_{ij}^{\alpha}\overline{R}_{\alpha kik;j}+\sum_{i,j,k,m,\alpha,\alpha'}h_{ij}^{\alpha}h_{kj}^{\alpha'}h_{im}^{\alpha}h_{mk}^{\alpha'}\\&&-\sum_{i,j,k,m,\alpha,\alpha'}h_{ij}^{\alpha}h_{kj}^{\alpha'}h_{mk}^{\alpha}h_{im}^{\alpha'}+\sum_{i,j,k,\alpha,\alpha'}h_{ij}^{\alpha}h_{ij}^{\alpha'}\overline{R}_{\alpha k \alpha'k}\\&&+\sum_{i,j,k,\alpha,\alpha'}h_{ij}^{\alpha}h_{kj}^{\alpha'}\overline{R}_{\alpha ki \alpha'}+\sum_{i,j,k,m,\alpha}h_{ij}^{\alpha}h_{jm}^{\alpha}\overline{R}_{mkik},
\end{eqnarray*}
therefore,

\begin{eqnarray*}
\sum_{i,j,k,\alpha}h_{ij}^{\alpha}
  (\overline{R}_{\alpha ijkk}+\overline{R}_{\alpha kikj})
&=& \sum_{i,j,k,\alpha}h_{ij}^{\alpha}
       (\overline{R}_{\alpha ijk;k}+\overline{R}_{\alpha kik;j}) \\[0.3em]
&&+ \sum_{i,j,k,m,\alpha,\alpha'}h_{ij}^{\alpha}h_{ik}^{\alpha'}h_{jm}^{\alpha}h_{mk}^{\alpha'}
   - \sum_{i,j,k,m,\alpha,\alpha'}h_{ij}^{\alpha}h_{ik}^{\alpha'}h_{mk}^{\alpha}h_{jm}^{\alpha'} \\[0.3em]
&&+ \sum_{i,j,k,\alpha,\alpha'}h_{ij}^{\alpha}h_{jk}^{\alpha'}\overline{R}_{\alpha i\alpha' k}
   + \sum_{i,j,k,\alpha,\alpha'}h_{ij}^{\alpha}h_{kk}^{\alpha'}\overline{R}_{\alpha ij \alpha'} \\[0.3em]
&&+ \sum_{i,j,k,m,\alpha}h_{ij}^{\alpha}h_{km}^{\alpha}\overline{R}_{mijk}
   + \sum_{i,j,k,m,\alpha,\alpha'}h_{ij}^{\alpha}h_{kj}^{\alpha'}h_{im}^{\alpha}h_{mk}^{\alpha'} \\[0.3em]
&&- \sum_{i,j,k,m,\alpha,\alpha'}h_{ij}^{\alpha}h_{kj}^{\alpha'}h_{mk}^{\alpha}h_{im}^{\alpha'}
   + \sum_{i,j,k,\alpha,\alpha'}h_{ij}^{\alpha}h_{ij}^{\alpha'}\overline{R}_{\alpha k\alpha' k} \\[0.3em]
&&+ \sum_{i,j,k,\alpha,\alpha'}h_{ij}^{\alpha}h_{kj}^{\alpha'}\overline{R}_{\alpha ki \alpha'}
   + \sum_{i,j,k,m,\alpha}h_{ij}^{\alpha}h_{jm}^{\alpha}\overline{R}_{mkik}.
\end{eqnarray*}
Since that $\mathbb{L}_{q}^{n+p}$ is {symmetric locally}, we have
\begin{eqnarray*}
\sum_{i,j,k,\alpha}h_{ij}^{\alpha}(\overline{R}_{\alpha ijk;k}+\overline{R}_{\alpha kik;j})=0,
\end{eqnarray*}
hence,

\begin{eqnarray*}
\sum_{i,j,k,\alpha}h_{ij}^{\alpha}
   (\overline{R}_{\alpha ijkk}+\overline{R}_{\alpha kikj})
&=& \sum_{i,j,k,m,\alpha,\alpha'}h_{ij}^{\alpha}h_{ik}^{\alpha'}h_{jm}^{\alpha}h_{mk}^{\alpha'}
   - \sum_{i,j,k,m,\alpha,\alpha'}h_{ij}^{\alpha}h_{ik}^{\alpha'}h_{mk}^{\alpha}h_{jm}^{\alpha'} \\[0.3em]
&&+ \sum_{i,j,k,\alpha,\alpha'}h_{ij}^{\alpha}h_{jk}^{\alpha'}\overline{R}_{\alpha i\alpha' k}
   + \sum_{i,j,k,\alpha,\alpha'}h_{ij}^{\alpha}h_{kk}^{\alpha'}\overline{R}_{\alpha ij \alpha'} \\[0.3em]
&&+ \sum_{i,j,k,m,\alpha}h_{ij}^{\alpha}h_{km}^{\alpha}\overline{R}_{mijk}
   + \sum_{i,j,k,m,\alpha,\alpha'}h_{ij}^{\alpha}h_{kj}^{\alpha'}h_{im}^{\alpha}h_{mk}^{\alpha'} \\[0.3em]
&&- \sum_{i,j,k,m,\alpha,\alpha'}h_{ij}^{\alpha}h_{kj}^{\alpha'}h_{mk}^{\alpha}h_{im}^{\alpha'}
   + \sum_{i,j,k,\alpha,\alpha'}h_{ij}^{\alpha}h_{ij}^{\alpha'}\overline{R}_{\alpha k\alpha' k} \\[0.3em]
&&+ \sum_{i,j,k,\alpha,\alpha'}h_{ij}^{\alpha}h_{kj}^{\alpha'}\overline{R}_{\alpha ki \alpha'}
   + \sum_{i,j,k,m,\alpha}h_{ij}^{\alpha}h_{jm}^{\alpha}\overline{R}_{mkik}.
\end{eqnarray*}
Now, observe that

\begin{eqnarray*}
\sum_{i,j,k,\alpha}h_{ij}^{\alpha}h_{kkij}^{\alpha}=n\sum_{i,j,\alpha}h_{ij}^{\alpha}H_{ij}^{\alpha}.
\end{eqnarray*}
Using the fact that $H_{kj}=H_{kj}^{n+p-q+1}$ and $H_{kj}^{\alpha}=0,\quad \alpha \neq n+p-q+1,$ we have

\begin{eqnarray*}
\sum_{i,j,k,\alpha}h_{ij}^{\alpha}h_{kkij}^{\alpha}=n\sum_{i,j}h_{ij}^{n+p-q+1}H_{ij}.
\end{eqnarray*}
Finally, we conclude that
\begin{eqnarray*}
\dfrac{1}{2}\Delta S&=&|\nabla B|^2+\sum_{i,j,k,\alpha}h_{ij}^{\alpha}(\overline{R}_{\alpha ijkk}+\overline{R}_{\alpha kikj})+\sum_{i,j,k,\alpha}h_{ij}^{\alpha}h_{kkij}^{\alpha}\\&&+\sum_{i,j,k,m,\alpha}h_{ij}^{\alpha}h_{km}^{\alpha}R_{mijk}+\sum_{i,j,k,m,\alpha}h_{ij}^{\alpha}h_{mi}^{\alpha}R_{mkjk},
\end{eqnarray*}
this is, 
$$
\begin{aligned}
\dfrac{1}{2}\Delta S \;=\;& |\nabla B|^2
   + \sum_{i,j,k,m,\alpha,\alpha'} h_{ij}^{\alpha}h_{ik}^{\alpha'}h_{jm}^{\alpha}h_{mk}^{\alpha'}
   - \sum_{i,j,k,m,\alpha,\alpha'} h_{ij}^{\alpha}h_{ik}^{\alpha'}h_{mk}^{\alpha}h_{jm}^{\alpha'} \\[6pt]
 &+ \sum_{i,j,k,\alpha,\alpha'} h_{ij}^{\alpha}h_{jk}^{\alpha'}\overline{R}_{\alpha i\alpha' k}
   + \sum_{i,j,k,\alpha,\alpha'} h_{ij}^{\alpha}h_{kk}^{\alpha'}\overline{R}_{\alpha ij \alpha'}
   + \sum_{i,j,k,m,\alpha} h_{ij}^{\alpha}h_{km}^{\alpha}\overline{R}_{mijk} \\[6pt]
 &+ \sum_{i,j,k,m,\alpha,\alpha'} h_{ij}^{\alpha}h_{kj}^{\alpha'}h_{im}^{\alpha}h_{mk}^{\alpha'}
   - \sum_{i,j,k,m,\alpha,\alpha'} h_{ij}^{\alpha}h_{kj}^{\alpha'}h_{mk}^{\alpha}h_{im}^{\alpha'}
   + \sum_{i,j,k,\alpha,\alpha'} h_{ij}^{\alpha}h_{ij}^{\alpha'}\overline{R}_{\alpha k\alpha' k} \\[6pt]
 &+ \sum_{i,j,k,\alpha,\alpha'} h_{ij}^{\alpha}h_{kj}^{\alpha'}\overline{R}_{\alpha ki \alpha'}
   + \sum_{i,j,k,m,\alpha} h_{ij}^{\alpha}h_{jm}^{\alpha}\overline{R}_{mkik}
   + n\sum_{i,j} h_{ij}^{n+p-q+1}H_{ij} \\[6pt]
 &+ \sum_{i,j,k,m,\alpha} h_{ij}^{\alpha}h_{km}^{\alpha}\overline{R}_{mijk}
   + \sum_{i,j,k,m,\alpha,\beta} h_{ij}^{\alpha}h_{km}^{\alpha}h_{mj}^{\beta}h_{ik}^{\beta}
   - \sum_{i,j,k,m,\alpha,\beta} h_{ij}^{\alpha}h_{km}^{\alpha}h_{mk}^{\beta}h_{ij}^{\beta} \\[6pt]
 &- \sum_{i,j,k,m,\alpha,\gamma} h_{ij}^{\alpha}h_{km}^{\alpha}h_{mj}^{\gamma}h_{ik}^{\gamma}
   + \sum_{i,j,k,m,\alpha,\gamma} h_{ij}^{\alpha}h_{km}^{\alpha}h_{mk}^{\gamma}h_{ij}^{\gamma}
   + \sum_{i,j,k,m,\alpha} h_{ij}^{\alpha}h_{mi}^{\alpha}\overline{R}_{mkjk} \\[6pt]
 &- n\sum_{i,j,m,\alpha} h_{ij}^{\alpha}h_{mi}^{\alpha}h_{mj}^{n+p-q+1}H
   - \sum_{i,j,k,m,\alpha,\beta} h_{ij}^{\alpha}h_{mi}^{\alpha}h_{mk}^{\beta}h_{kj}^{\beta}
   + \sum_{i,j,k,m,\alpha,\gamma} h_{ij}^{\alpha}h_{mi}^{\alpha}h_{mk}^{\gamma}h_{kj}^{\gamma}.
\end{aligned}
$$
Also note that
\begin{eqnarray*}
\sum_{i,j,k,m,\alpha,\alpha'}h_{ij}^{\alpha}h_{ik}^{\alpha'}h_{jm}^{\alpha}h_{mk}^{\alpha'}&-&\sum_{i,j,k,m,\alpha,\alpha'}h_{ij}^{\alpha}h_{ik}^{\alpha'}h_{mk}^{\alpha}h_{jm}^{\alpha'}\\&+&\sum_{i,j,k,m,\alpha,\alpha'}h_{ij}^{\alpha}h_{kj}^{\alpha'}h_{im}^{\alpha}h_{mk}^{\alpha'}-\sum_{i,j,k,m,\alpha,\alpha'}h_{ij}^{\alpha}h_{kj}^{\alpha'}h_{mk}^{\alpha}h_{im}^{\alpha'}\\
&=&\dfrac{1}{2}\sum_{\alpha,\alpha'}N(h^{\alpha}h^{\alpha'}-h^{\alpha'}h^{\alpha})+\sum_{\alpha,\alpha'}tr(h^{\alpha}h^{\alpha'}h^{\alpha'}h^{\alpha})-\sum_{\alpha,\alpha'}tr(h^{\alpha}h^{\alpha'})^2,
\end{eqnarray*}
on the other hand,
$$
\begin{aligned}
\sum_{i,j,k,m,\alpha,\beta} h_{ij}^{\alpha}h_{km}^{\alpha}h_{mj}^{\beta}h_{ik}^{\beta}
&- \sum_{i,j,k,m,\alpha,\beta} h_{ij}^{\alpha}h_{km}^{\alpha}h_{mk}^{\beta}h_{ij}^{\beta} - \sum_{i,j,k,m,\alpha,\beta} h_{ij}^{\alpha}h_{mi}^{\alpha}h_{mk}^{\beta}h_{kj}^{\beta} \\[6pt]
& = \sum_{\alpha,\beta} \operatorname{tr}\!\big((h^{\alpha}h^{\beta})^2\big)
   - \sum_{\alpha,\beta} \big[\operatorname{tr}(h^{\alpha}h^{\beta})\big]^2
   - \sum_{\alpha,\beta} \operatorname{tr}(h^{\alpha}h^{\beta}h^{\beta}h^{\alpha}),
\end{aligned}
$$
and still

\begin{eqnarray*}
\sum_{i,j,k,m,\alpha,\gamma}h_{ij}^{\alpha}h_{km}^{\alpha}h_{mk}^{\gamma}h_{ij}^{\gamma}&-&\sum_{i,j,k,m,\alpha,\gamma}h_{ij}^{\alpha}h_{km}^{\alpha}h_{mj}^{\gamma}h_{ik}^{\gamma}+\sum_{i,j,k,m,\alpha,\gamma}h_{ij}^{\alpha}h_{mi}^{\alpha}h_{mk}^{\gamma}h_{kj}^{\gamma}\\
&=&\sum_{\alpha,\gamma}[tr(h^{\alpha}h^{\gamma})]^2-\sum_{\alpha,\gamma}tr(h^{\alpha}h^{\gamma})^2 +\sum_{\alpha,\gamma}tr(h^{\alpha}h^{\gamma}h^{\gamma}h^{\alpha}).
\end{eqnarray*}
Therefore of these last four equalities, we obtain

\begin{equation}\label{eq:DeltaS}
\begin{aligned}
\dfrac{1}{2}\Delta S \;=\;& |\nabla B|^2
  + \sum_{i,j,k,\alpha,\alpha'} h_{ij}^{\alpha}h_{jk}^{\alpha'}\overline{R}_{\alpha i \alpha' k}
  + \sum_{i,j,k,\alpha,\alpha'} h_{ij}^{\alpha}h_{kk}^{\alpha'}\overline{R}_{\alpha ij \alpha'}
  + \sum_{i,j,k,m,\alpha} h_{ij}^{\alpha}h_{km}^{\alpha}\overline{R}_{mijk} \\
&+ \sum_{i,j,k,\alpha,\alpha'} h_{ij}^{\alpha}h_{ij}^{\alpha'}\overline{R}_{\alpha k \alpha' k}
  + \sum_{i,j,k,\alpha,\alpha'} h_{ij}^{\alpha}h_{kj}^{\alpha'}\overline{R}_{\alpha ki \alpha'}
  + \sum_{i,j,k,m,\alpha} h_{ij}^{\alpha}h_{jm}^{\alpha}\overline{R}_{mkik} \\
&+ n\sum_{i,j} h_{ij}^{\,n+p-q+1}H_{ij}
  + \sum_{i,j,k,m,\alpha} h_{ij}^{\alpha}h_{km}^{\alpha}\overline{R}_{mijk}
  + \sum_{i,j,k,m,\alpha} h_{ij}^{\alpha}h_{mi}^{\alpha}\overline{R}_{mkjk} \\
&+ n\sum_{i,j,m,\alpha} h_{ij}^{\alpha}h_{mi}^{\alpha}h_{mj}^{\,n+p-q+1}H
  + \dfrac{1}{2}\sum_{\alpha,\alpha'} N\!\big(h^{\alpha}h^{\alpha'}-h^{\alpha'}h^{\alpha}\big)
  + \sum_{\alpha,\alpha'} \operatorname{tr}\!\big(h^{\alpha}h^{\alpha'}h^{\alpha'}h^{\alpha}\big) \\
&- \sum_{\alpha,\alpha'} \big(\operatorname{tr}(h^{\alpha}h^{\alpha'})\big)^2
  + \sum_{\alpha,\beta} \big(\operatorname{tr}(h^{\alpha}h^{\beta})\big)^2
  - \sum_{\alpha,\beta} \big[\operatorname{tr}(h^{\alpha}h^{\beta})\big]^2 \\
&- \sum_{\alpha,\beta} \operatorname{tr}\!\big(h^{\alpha}h^{\beta}h^{\beta}h^{\alpha}\big)
  + \sum_{\alpha,\gamma} \big[\operatorname{tr}(h^{\alpha}h^{\gamma})\big]^2
  - \sum_{\alpha,\gamma} \big(\operatorname{tr}(h^{\alpha}h^{\gamma})\big)^2 \\
&+ \sum_{\alpha,\gamma} \operatorname{tr}\!\big(h^{\alpha}h^{\gamma}h^{\gamma}h^{\alpha}\big)\,.
\end{aligned}
\end{equation}
Moreover, since 

\begin{eqnarray}
\sum_{\alpha,\alpha'}tr(h^{\alpha}h^{\beta}h^{\beta}h^{\alpha})-\sum_{\alpha,\beta}tr(h^{\alpha}h^{\beta})^2=\dfrac{1}{2}\sum_{\alpha,\alpha'}N(h^{\alpha}h^{\alpha'}-h^{\alpha'}h^{\alpha}).
\end{eqnarray}
Thus, from the preceding construction, we derive the following Simons-type formula. The following result generalizes the Simons-type formula $(2.10)$ of Gomes et. al in \cite{nazareno} and Lemma 2 established by Araújo et. al \cite{Araujo:17}.

\begin{lemma}\label{formula simons}
Let $X:M^{n}\looparrowright \mathbb{L}^{n+p}_{q}$ ($1\leq q\leq p$) be an submanifold with flat normal bundle and parallel normalized mean curvature vector field. Then, we get
\begin{equation}\label{simons_identity}
\begin{aligned}
\dfrac{1}{2}\Delta S \;=\;& |\nabla B|^2
  + 2\Bigg(
      \sum_{i,j,k,m,\alpha} h_{ij}^{\alpha}h_{km}^{\alpha}\overline{R}_{mijk}
      + \sum_{i,j,k,m,\alpha} h_{ij}^{\alpha}h_{jm}^{\alpha}\overline{R}_{mkik}
    \Bigg) \\[6pt]
&+ \sum_{i,j,k,\alpha,\alpha'} h_{ij}^{\alpha}h_{jk}^{\alpha'}\overline{R}_{\alpha i \alpha' k}
  + \sum_{i,j,k,\alpha,\alpha'} h_{ij}^{\alpha}h_{kk}^{\alpha'}\overline{R}_{\alpha ij \alpha'}
  + \sum_{i,j,k,\alpha,\alpha'} h_{ij}^{\alpha}h_{ij}^{\alpha'}\overline{R}_{\alpha k \alpha' k} \\[6pt]
&+ \sum_{i,j,k,\alpha,\alpha'} h_{ij}^{\alpha}h_{kj}^{\alpha'}\overline{R}_{\alpha ki \alpha'}
  + n\sum_{i,j} h_{ij}^{\,n+p-q+1}H_{ij}
  - nH\sum_{i,j,m,\alpha} h_{ij}^{\alpha}h_{mi}^{\alpha}h_{mj}^{\,n+p-q+1} \\[6pt]
&+ \sum_{\alpha,\alpha'} N(h^{\alpha}h^{\alpha'}-h^{\alpha'}h^{\alpha})
  - \sum_{\alpha,\beta} \big[\operatorname{tr}(h^{\alpha}h^{\beta})\big]^2
  - \dfrac{1}{2}\sum_{\alpha,\beta} N(h^{\alpha}h^{\beta}-h^{\beta}h^{\alpha}) \\[6pt]
&+ \sum_{\alpha,\gamma} \big[\operatorname{tr}(h^{\alpha}h^{\gamma})\big]^2
  + \dfrac{1}{2}\sum_{\alpha,\gamma} N(h^{\alpha}h^{\gamma}-h^{\gamma}h^{\alpha}) \,.
\end{aligned}
\end{equation}
where $N(A)=tr(AA^{t}),$ for all matrix $A=(a_{ij}).$
\end{lemma}

Otherwise, we can rewrite the above identity with the following inequality.

\begin{lemma}\label{simons desigualdade}
Let $X:M^{n}\looparrowright \mathbb{L}^{n+p}_{q}$ ($1\leq q\leq p$) be an submanifold with flat normal bundle and parallel normalized mean curvature vector field. Then, we get
\begin{equation}\label{simons_inequality}
\begin{aligned}
\dfrac{1}{2}\Delta S \; \geq \;& |\nabla B|^2
  + 2\Bigg(
      \sum_{i,j,k,m,\alpha} h_{ij}^{\alpha}h_{km}^{\alpha}\overline{R}_{mijk}
      + \sum_{i,j,k,m,\alpha} h_{ij}^{\alpha}h_{jm}^{\alpha}\overline{R}_{mkik}
    \Bigg) \\[6pt]
&+ \sum_{i,j,k,\alpha,\alpha'} h_{ij}^{\alpha}h_{jk}^{\alpha'}\overline{R}_{\alpha i \alpha' k}
  + \sum_{i,j,k,\alpha,\alpha'} h_{ij}^{\alpha}h_{kk}^{\alpha'}\overline{R}_{\alpha ij \alpha'}
  + \sum_{i,j,k,\alpha,\alpha'} h_{ij}^{\alpha}h_{ij}^{\alpha'}\overline{R}_{\alpha k \alpha' k} \\[6pt]
&+ \sum_{i,j,k,\alpha,\alpha'} h_{ij}^{\alpha}h_{kj}^{\alpha'}\overline{R}_{\alpha ki \alpha'}
  + n\sum_{i,j} h_{ij}^{\,n+p-q+1}H_{ij}
  - nH\sum_{i,j,m,\alpha} h_{ij}^{\alpha}h_{mi}^{\alpha}h_{mj}^{\,n+p-q+1} \\[6pt]
& - \sum_{\alpha,\beta} \big[\operatorname{tr}(h^{\alpha}h^{\beta})\big]^2
+ \sum_{\alpha,\gamma} \big[\operatorname{tr}(h^{\alpha}h^{\gamma})\big]^2
  \,.
\end{aligned}
\end{equation}
    
\end{lemma}

\section{Curvature Constraints, Examples, and Key Lemmas}\label{sec:key lemmas}

Proceeding with the context of the previous section and inspired by the configuration assumed by Nishikawa in \cite{Nishikawa: 84}, along this work we will assume that there exist constants $c_1,c_2$ and $c_3$ such that the sectional curvature $\overline{K}$ and the curvature tensor $\overline{R}$ of the ambient space
\begin{eqnarray*}
\mathbb{L}_{q}^{n+p}=\langle e_1,e_2,\cdots, e_n,e_{n+1},\cdots,e_{n+p-q},e_{n+p-q+1},\cdots,e_{n+p}\rangle,
\end{eqnarray*}
satisfies the following constraints:

\begin{equation}\label{c1}
\overline{K}(x_i,x_{\alpha})=\dfrac{c_1}{n},\quad x_i\in \langle e_1,\cdots, e_n\rangle\quad\mbox{and}\quad x_{\alpha}\in \langle e_{n+1},\cdots,e_{n+p}\rangle,
\end{equation}

\begin{equation}\label{c2}
\overline{K}(x_i,x_j)\geq c_2,\quad x_i,x_j\in \langle e_1,\cdots,e_n\rangle,
\end{equation}

\begin{equation}\label{R barra}
\langle \overline{R}(x_{\alpha},x_i)x_{\alpha'},x_i\rangle=0,\quad x_i\in \langle e_1,\cdots, e_n\rangle\quad\mbox{and}\quad x_{\alpha},x_{\alpha'}\in \langle e_{n+1},\cdots,e_{n+p}\rangle,
\end{equation}

\begin{equation}\label{c3}
\overline{K}(x_{\alpha},x_{\alpha'})=\dfrac{c_3}{p},\quad x_{\alpha},x_{\alpha'}\in \langle e_{n+1},\cdots,e_{n+p}\rangle.
\end{equation}
If $\mathbb{L}_{q}^{n+p}$ satisfies conditions above, then
\begin{eqnarray*}
\overline{R}=\sum_{i,j,i,j}\overline{R}_{ijij}+D(c_1,c_3),
\end{eqnarray*}
where $D(c_1,c_3)$ is a constant.
Then, it is well known that the scalar curvature of a locally symmetric Lorentz
space is constant. Consequently, $\displaystyle\sum_{i,j}\overline{R}_{ijij}$ is a constant naturally attached
to a locally symmetric Lorentz space satisfying conditions \eqref{c1} and \eqref{c3}.
For sake of simplicity, in the course of this work we will denote the constant $\dfrac{1}{n(n-1)}\displaystyle\sum_{i,j}\overline{R}_{ijij}$ by $\overline{\mathcal{R}}$. In order to establish our main results, we devote this
section to present some auxiliary lemmas. 

\begin{remark}
The curvature conditions \eqref{c1} and \eqref{c3}, are natural extensions for higher codimensions of conditions assumed by Nishikawa \cite{Nishikawa: 84} in context of hypersurfaces. When the ambient manifold $\mathbb{L}_{q}^{n+p}$ has constant sectional curvature $c$, then it satisfies conditions \eqref{c1}, \eqref{c2}, \eqref{R barra} and \eqref{c3}. On the other hand, the next example gives us a situation where the curvature conditions \eqref{c1}, \eqref{c2} and \eqref{c3} are satisfied but the ambient space has not constant sectional curvature.
\end{remark}

\begin{example}
Let $\mathbb{L}_{q}^{n+p}=\mathbb{R}_{q}^{p}\times\mathbb{S}^{n}$ be a semi-Riemannian space, where $\mathbb{R}_{q}^{p}$ stands for the $p$-dimensional semi-Euclidean space of index $q$ and $\mathbb{S}^{n}$ is the $n$-dimensional unit Euclidean sphere. We consider the spacelike submanifold $M^{n}=\{0\}\times\mathbb{S}^{n}$ of $\mathbb{L}_{q}^{n+p}$. Taking into account that the normal bundle of $M^{n}$ is equipped with $p$ linearly independent timelike vector fields $\xi^{1},\xi^{2},\ldots,\xi^{p}$, it is not difficult to verify that the sectional curvature $\overline{K}$ of $\mathbb{L}_{q}^{n+p}$ satisfies
\begin{eqnarray}\label{exemplo1}
\overline{K}\left(\xi_{i},X\right)&=&\langle R_{\mathbb{R}_{q}^{p}}(\xi^{i},X)\xi^{i},X\rangle_{\mathbb{R}_{q}^{p}}+\langle R_{\mathbb{S}^{n}}(0,U)0,U\rangle_{\mathbb{S}^{n}}=0,
\end{eqnarray}
for each $i\in\{1,\ldots,p\}$, where $R_{\mathbb{R}_{q}^{p}}$ and $R_{\mathbb{S}^{n}}$ denote the curvature tensors of $\mathbb{R}_{q}^{p}$ and $\mathbb{S}^{n}$, respectively, $\xi_{i}=(\xi^{i},0)\in T^{\perp}M$ and $X=(0,X_{2})\in TM$ with $\langle \xi_{i},\xi_{i}\rangle=\langle X,X\rangle=1$.
On the other hand, by a direct computation we obtain
\begin{equation}\label{eqaux:1}
\overline{K}(X,Y)=\langle R_{\mathbb{R}_{q}^{p}}(0,0)0,0\rangle_{\mathbb{R}_{q}^{p}}+\langle R_{\mathbb{S}^{n}}(X_{2},Y_{2})X_{2},Y_{2}\rangle_{\mathbb{S}^{n}}
\end{equation}
for every $X=(0,X_{2}), Y=(0,Y_{2})\in TM$ such that $\langle X,Y\rangle=0,\langle X,X\rangle=\langle Y,Y\rangle=1$. Consequently, since
\begin{eqnarray*}
\langle X_{2},Y_{2}\rangle=0,\langle X_{2},X_{2}\rangle=\langle Y_{2},Y_{2}\rangle=1,
\end{eqnarray*}
from~\eqref{eqaux:1} we get
\begin{eqnarray}\label{exemplo2}
\overline{K}(X,Y)=|X_{2}|^{2}|Y_{2}|^{2}-\langle X_{2},Y_{2}\rangle^{2}=1.
\end{eqnarray}
Moreover, we have that
\begin{eqnarray}\label{exemplo3}
\overline{K}(\xi_{i},\xi_{j})=0,\quad\mbox{for all}\quad i,j\in \{1,\ldots,p\}
\end{eqnarray}
and
\begin{eqnarray}\label{exemplo4}
\langle \overline{R}(\xi_{i}, X)\xi_{j},X\rangle=0,\quad\mbox{for all}\quad i,j\in \{1,\ldots,p\}.
\end{eqnarray}
We observe from ~\eqref{exemplo1},~\eqref{exemplo2}, ~\eqref{exemplo3} and ~\eqref{exemplo4} that the curvature constraints ~\eqref{c1}-\eqref{c3} are satisfied with $c_{1}=c_{3}=0$, $c_{2}=1$ and $c=\dfrac{c_{1}}{n}+2c_{2}=2$. Furthermore, we also note that the ambient space $\mathbb{L}_{q}^{n+p}=\mathbb{R}_{q}^{p}\times\mathbb{S}^{n}$ is conformally flat (see, for instance, Chapter $7$ of~\cite{Dajczer:90}).
\end{example}

In order to establish our main results, we also devote this section to presenting a collection of auxiliary lemmas. In particular, the next result extends Lemma 2 of Gomes N.J. et al in \cite{nazareno} and Lemma 1 of \cite{Araujo:17} in the special case $q=p$.

\begin{lemma}\label{desigualdade_BeH}
Let $X:M^n\looparrowright \mathbb{L}^{n+p}_{q}$ be a complete spacelike LW-submanifold satisfying conditions \eqref{c1}-\eqref{c3} and \eqref{equacaoRS}. If the second fundamental form of $M^n$ is locally timelike and
\begin{equation}\label{H1}
(n-1)a^2+4n(\overline{\mathcal{R}}-b)\geq 0.
\end{equation}
Then, 
\begin{eqnarray}\label{desigualdade}
|\nabla B|^2 \geq n^2|\nabla H|^2.
\end{eqnarray}
Moreover, is the equality holds in ~\eqref{desigualdade} on $M^{n}$, then $H$ is constant on $M^{n}.$
\end{lemma}

\begin{proof} 

Since we are supposing that $R=aH+b$ and $\mathbb{L}_{q}^{n+p}$ satisfies the conditions $\eqref{c1}$-$\eqref{c3}$ and \eqref{equacaoRS}, we have
\begin{eqnarray*}
S_1-n^2\sum_{\beta}(H^{\beta})^2+n(n-1)R=\sum_{i,j}\overline{R}_{ijij}-n^2\sum_{\gamma}(H^{\gamma})^2+S_2.
\end{eqnarray*}
On the other hand, as the second fundamental form of $M^n$ is locally timelike, we have that $H^{\beta}=0,$ for $n+1\leq \beta\leq n+p-q$, that is, 
$$H^2=\sum_{\alpha}(H^{\alpha})^2=\sum_{\gamma}(H^{\gamma})^2.$$ 
Moreover, we have
\begin{eqnarray*}
S_1+n(n-1)R=\sum_{i,j}\overline{R}_{ijij}-n^2H^2+S_2,
\end{eqnarray*}
using the fact that $R=aH+b$, we have
$$
S_1+n(n-1)(aH+b)=\sum_{i,j}\overline{R}_{ijij}-n^2H^2+S_2,
$$
organizing the terms, we obtain
$$
S_2-S_1=n(n-1)(aH+b)-\sum_{i,j}\overline{R}_{ijij}+n^2H^2,
$$
or yet,
$$
\sum_{i,j,\gamma}(h_{ij}^{\gamma})^2-\sum_{i,j,\beta}(h_{ij}^{\beta})^2=n(n-1)aH+n(n-1)b-\sum_{i,j}\overline{R}_{ijij}+n^2H^2,
$$
thus, we get
\begin{equation}\label{L0}
2\left(\sum_{i,j,\gamma}h_{ij}^{\gamma}h_{ijk}^{\gamma}-\sum_{i,j,\beta}h_{ij}^{\beta}h_{ijk}^{\beta}\right)=(2n^2H+n(n-1)a)\cdot H_k,
\end{equation}
where $H_k$ stands for the $k$-th component of $\nabla H.$ With a quick calculation we get
\begin{eqnarray*}
4\sum_{k}\left(\sum_{i,j,\gamma}h_{ij}^{\gamma}h_{ijk}^{\gamma}-\sum_{i,j,\beta}h_{ij}^{\beta}h_{ijk}^{\beta}\right)^2=(2n^2H+n(n-1)a)^2\cdot |\nabla H|^2.
\end{eqnarray*}
Now note that 

\begin{eqnarray*}
4S|\nabla B|^2&=&4\sum_{i,j,\alpha}(h_{ij}^{\alpha})^2\sum_{i,j,k,\alpha}(h_{ijk}^{\alpha})^2=4\left(\sum_{i,j,\beta}(h_{ij}^{\beta})^2+\sum_{i,j,\gamma}(h_{ij}^{\gamma})^2\right)\cdot \left(\sum_{i,j,k,\beta}(h_{ijk}^{\beta})^2+\sum_{i,j,k,\gamma}(h_{ijk}^{\gamma})^2\right)\\
&=&4\sum_{i,j,\beta}(h_{ij}^{\beta})^2\cdot \sum_{i,j,k,\beta}(h_{ijk}^{\beta})^2+4\sum_{i,j,\beta}(h_{ij}^{\beta})^2\cdot \sum_{i,j,k,\gamma}(h_{ijk}^{\gamma})^2\\
&+&4\sum_{i,j,\gamma}(h_{ij}^{\gamma})^2\cdot \sum_{i,j,k,\beta}(h_{ijk}^{\beta})^2+4\sum_{i,j,\gamma}(h_{ij}^{\gamma})^2\cdot \sum_{i,j,k,\gamma}(h_{ijk}^{\gamma})^2,
\end{eqnarray*}
consequently, using Cauchy-Schwarz inequality, we obtain that
\begin{eqnarray*}
4S|\nabla B|^2&\geq& 4\sum_k\left(\sum_{i,j,\beta}h_{ij}^{\beta}h_{ijk}^{\beta}\right)^2+4\sum_{k}\left(\sum_{i,j,\beta,\gamma}h_{ij}^{\beta}h_{ijk}^{\gamma}\right)^2\\
&&+4\sum_k\left(\sum_{i,j,\beta,\gamma}h_{ij}^{\gamma}h_{ijk}^{\beta}\right)^2+4\sum_k\left(\sum_{i,j,\gamma}h_{ij}^{\gamma}h_{ijk}^{\gamma}\right)^2\\
&=&4\sum_k\left[\left(\sum_{i,j,\gamma}h_{ij}^{\gamma}h_{ijk}^{\gamma}-\sum_{i,j,\beta}h_{ij}^{\beta}h_{ijk}^{\beta}\right)^2+2\sum_{i,j,\beta,\gamma}h_{ij}^{\gamma}h_{ijk}^{\gamma}h_{ij}^{\beta}h_{ijk}^{\beta}\right]\\
&&+4\sum_k\left[\left(\sum_{i,j,\beta,\gamma}h_{ij}^{\beta}h_{ijk}^{\gamma}\right)^2+\left(\sum_{i,j,\beta,\gamma}h_{ij}^{\gamma}h_{ijk}^{\beta}\right)^2\right]\\
&=&4\sum_k\left[\left(\sum_{i,j,\gamma}h_{ij}^{\gamma}h_{ijk}^{\gamma}-\sum_{i,j,\beta}h_{ij}^{\beta}h_{ijk}^{\beta}\right)^2\right]\\
&&+4\sum_k\left[\left(\sum_{i,j,\beta,\gamma}h_{ij}^{\beta}h_{ijk}^{\gamma}+\sum_{i,j,\beta,\gamma}h_{ij}^{\gamma}h_{ijk}^{\beta}\right)^2\right].
\end{eqnarray*}
Therefore,
\begin{equation}\label{L1}
4S|\nabla B|^2\geq 4\sum_k\left(\sum_{i,j,\gamma}h_{ij}^{\gamma}h_{ijk}^{\gamma}-\sum_{i,j,\beta}h_{ij}^{\beta}h_{ijk}^{\beta}\right)^2=(2n^2H+n(n-1)a)^2 |\nabla H|^2.
\end{equation}
On the other hand, since $R=aH+ b$, from equation \eqref{equacaoRS} we easily see that
\begin{equation}\label{L2}
(2n^2H + n(n-1)a)^2 = n^2(n-1) [ (n-1)a^2 + 4n(\overline{\mathcal{R}}-b) ] + 4n^2 S,
\end{equation}
from \eqref{L1} and \eqref{L2} and taking account your assumption \eqref{H1}, we obtain
\begin{equation}
S|\nabla B|^2 \geq S n^2|\nabla H|^2.
\end{equation}
In this case, either $S=0$ and $|\nabla B|^2 = n^2|\nabla H|^2 = 0$ or $|\nabla B|^2 \geq n^2|\nabla H|^2.$ Now, suppose that $|\nabla B|^2 = n^2 |\nabla H|^2$. If $(n-1)a^2 + 4n(\overline{\mathcal{R}} - b) > 0$, from \eqref{L1} and \eqref{L2} we have
\begin{equation}
4S|\nabla B|^2 \geq n^2(n-1)[ (n-1)a^2 + 4n(\overline{\mathcal{R}} - b) ] + 4n^2S|\nabla H|^2 > 4n^2S|\nabla H|^2,
\end{equation}
and thus, we obtain that $H$ is constant. If $(n-1)a^2 + 4n(\overline{\mathcal{R}} - b) = 0$, then from \eqref{L2}
\begin{equation}\label{L3}
(2n^2H + n(n-1)a)^2 - 4n^2S = 0,
\end{equation}
this together with \eqref{L0} forces that
\begin{equation}\label{L4}
S^2_k = 4n^2SH_k^2, \quad k = 1, \cdots, n,
\end{equation}
where $S_k$ stands for the $k$-th component of $\nabla S$. Since the equality in \eqref{L1} holds, there exists a real function $c_k$ on $M^n$ such that 
\begin{equation}\label{L4}
h_{ijk}^{n+p-q+1} = c_k h_{ij}^{n+p-q+1} \quad \text{and} \quad h_{ijk}^{\alpha} = c_k h_{ij}^\alpha, \quad \alpha > n+p-q+1, \quad  i,j,k=1, \cdots, n.
\end{equation}
Taking the sun on both sides of equation \eqref{L4} with respect to $i=j$, we get
\begin{equation}\label{L5}
H_k = c_k H, \quad H^\alpha_k = 0, \quad \alpha > n+p-q+1; \quad k=1, \cdots, n.
\end{equation}
From second equation of \eqref{L5} we can see that $e_{n+p-q+1}$ is parallel. It follows from \eqref{L4} that
\begin{equation}\label{L6}
S_k = 2\left(\sum_{i,j,\gamma}h_{ij}^{\gamma}h_{ijk}^{\gamma}-\sum_{i,j,\beta}h_{ij}^{\beta}h_{ijk}^{\beta}\right) = 2c_kS, \quad k = 1, \cdots, n,
\end{equation}
multiplying both sides of equation \eqref{L6} by $H$ and using \eqref{L5} we have
\begin{equation}\label{L7}
HS_k = 2H_kS, \quad k=1, \cdots, n,    
\end{equation}
it follows from \eqref{L4} and \eqref{L7} that
\begin{equation}
H^2_kS = H^2_k n^2H^2, \cdots, k=1,\cdots, n.
\end{equation}
Therefore we conclude that
\begin{equation}\label{L8}
|\nabla H|^2 (S - n^2H^2) = 0.    
\end{equation}
We suppose that $H$ is not constant on $M^n$. In this case, $|\nabla H|$ is not vanishing identically on $M^n$. Let us denote
$$
M_0 = \{ x \in M ~ ; ~ |\nabla H|> 0 \} \quad \text{and} \quad T = S - n^2H^2,
$$
it follows from \eqref{L8} that $M_0$ is open in $M$ and $T=0$ over $M_0$. From the continuity of $T$, we have that $T=0$ on the clousure $cl(M_0)$ of $M_0$. If $M - cl(M_0) \neq \emptyset$, then $H$ is constant in $M - cl(M_0)$. It follows that $S$ is constant and thus $T$ is constant on $M-cl(M_0)$. From the continuity of $T$, we have that $T=0$ and hence $S=n^2H^2$ on $M^n$. It follows that $H$ is constant on $M^n$, which contradicts the assumption. Therefore, this complete the proof.

\end{proof}

For what follows, we choose the normal vector field $e_{n+p-q+1}$ to be aligned with the mean curvature vector field, that is,
$$
e_{n+p-q+1} = \frac{h}{H}.
$$
Within this framework, we introduce the following symmetric tensor:
$$
\Phi = \sum_{i,j,\alpha} \Phi_{ij}^{\alpha} \, \omega_i \otimes \omega_j \, e_{\alpha},
$$
where the components are given by
$$
\Phi_{ij}^{n+p-q+1} = h_{ij}^{n+p-q+1} - H \delta_{ij}
\quad \text{and} \quad
\Phi_{ij}^{\alpha} = h_{ij}^{\alpha}, \quad n+2 \leq \alpha \leq n+p.
$$
The squared length of the tensor $\Phi$ is defined by
$$
|\Phi|^2 = \sum_{i,j,\alpha} \left( \Phi_{ij}^{\alpha} \right)^2.
$$
Moreover, one easily verifies that
$$
|\Phi|^2 = S - nH^2 = S_1 + S_2 - nH^2,
$$
where
$$
S_1 = \sum_{i,j,\beta} \left( h_{ij}^{\beta} \right)^2
\quad \text{and} \quad
S_2 = \sum_{i,j,\gamma} \left( h_{ij}^{\gamma} \right)^2.
$$
In addition, we have the following relation:
$$
n(n-1)(aH + b)
= \sum_{i,j} \overline{R}_{ijij}
- n^2 \left( \sum_{\gamma} (H^{\gamma})^2 - \sum_{\beta} (H^{\beta})^2 \right)
+ (S_2 - S_1).
$$
Since $\sum_{\beta} (H^{\beta})^2 = 0$, it follows that
\begin{equation}\label{eq:main_relation}
n(n-1)aH + n(n-1)b
= \sum_{i,j} \overline{R}_{ijij}
- n^2 H^2 + (S_2 - S_1).
\end{equation}

In order to study LW-submanifolds, we will consider, for
each $a\in \mathbb{R}$, an appropriated Cheng-Yau’s modified operator, which is given
by
\begin{equation}\label{operadorL}
\mathcal{L} :=\square+\dfrac{n-1}{2}a\Delta,
\end{equation}
where, acoording to ~\cite{Cheng-Yau:77}, the square operator is defined by

\begin{equation}
\square f=\sum_{i,j}(nH\delta_{ij}-h_{ij}^{n+p-q+1})f_{ij},
\end{equation}
for each $f\in C^{\infty}(M),$ and the normal vector field $e_{n+p-q+1}$ is taken in the direction of the mean curvature vector field, that is, $e_{n+p-q+1}=\dfrac{h}{H}.$

The next lemma gives sufficient conditions to guarantee the elipticity of the operator $\mathcal{L}$, and it is an extension of Lemma 3 of Gomes et al in \cite{nazareno} and Lemma $3.2$ of de Lima in \cite{Lima: 13B}.

\begin{lemma}\label{L eliptico}
Let $X:M^n\looparrowright \mathbb{L}^{n+p}_{q}$ be a complete spacelike LW-submanifold with parallel normalized mean curvature vector field and flat normal bundle satisfying conditions \eqref{c1}-\eqref{c3}. If $b<\overline{\mathcal{R}}$, then $\mathcal{L}$ is elliptic.
\end{lemma}

\begin{proof}
Let us consider two cases: 

\noindent$(i)$ Case $a=0$. Since $R=b<\overline{\mathcal{R}}$, from Equation ~\eqref{equacaoRS} if we choose a (local) orthonormal frame $\{e_{i}\}$ on $M^{n}$ such that $h_{ij}^{n+p-q+1}=\lambda_{i}\delta_{ij}$, we have that $\sum_{i<j}\lambda_{i}\lambda_{j}>0$. Consequently,
\begin{eqnarray}
n^{2}H^{2}=\sum_{i}\lambda_{i}^{2}+2\sum_{i<j}\lambda_{i}\lambda_{j}>\lambda_{i}^{2}
\end{eqnarray}
for every $i=1,\ldots,n$ and, hence, we have that $nH-|\lambda_{i}|>0$ for every $i$.
Therefore, in this case, we conclude that $\mathcal{L}$ is elliptic.

\noindent $(ii)$ Case $a\neq 0$. From Equation ~\eqref{equacaoRS} we get that
\begin{eqnarray}\label{a}
a=\frac{1}{n(n-1)H}\left(S-n^{2}H^{2}+n(n-1)\overline{\mathcal{R}}-n(n-1)b\right).
\end{eqnarray}
For any $i$, from ~\eqref{a} we have
$$
nH-\lambda_{i}^{n+p-q+1}+\frac{n-1}{2}a=nH-\lambda_{i}^{n+p-q+1}+\frac{1}{2nH}\left(S-n^{2}H^{2}+n(n-1)(\overline{\mathcal{R}}-b)\right)
$$
\begin{equation}\label{**}
=\left(\frac{1}{2}(nH)^{2}-nH\lambda_{i}^{n+p-q+1}+\frac{1}{2}S+\frac{1}{2}n(n-1)(\overline{\mathcal{R}}-b)\right)(nH)^{-1}.
\end{equation}
Since $\displaystyle\sum_{j}\lambda_{j}^{n+p-q+1}=nH$ and $S\geq \displaystyle\sum_{j}(\lambda_{j}^{n+p-q+1})^{2}$, from ~\eqref{**} we have
\begin{eqnarray}
nH-\lambda_{i}^{n+p-q+1}+\frac{n-1}{2}a\!\!\!&\geq&\!\!\! \left\{\frac{1}{2}\left(\sum_{j}\lambda_{j}^{n+p-q+1}\right)^{2}-\lambda_{i}^{n+p-q+1}\sum_{j}\lambda_{j}^{n+p-q+1}+\frac{1}{2}\sum_{j}(\lambda_{j}^{n+p-q+1})^{2}\right\}(nH)^{-1}\nonumber\\
&&+\frac{1}{2}n(n-1)(\overline{\mathcal{R}}-b)(nH)^{-1}\nonumber\\
&=&\!\!\!\left\{\sum_{j}(\lambda_{j}^{n+p-q+1})^{2}+\frac{1}{2}\sum_{l\neq j}\lambda_{l}^{n+p-q+1}\lambda_{j}^{n+p-q+1}-\lambda_{i}^{n+p-q+1}\sum_{j}\lambda_{j}^{n+p-q+1}\right\}(nH)^{-1}\\
&&+\frac{1}{2}n(n-1)(\overline{\mathcal{R}}-b)(nH)^{-1}\nonumber\\
&=&\!\!\!\left\{\sum_{i\neq j}(\lambda_{j}^{n+p-q+1})^{2}+\frac{1}{2}\sum_{l\neq j, l,j\neq i}\lambda_{l}^{n+p-q+1}\lambda_{j}^{n+p-q+1}+\frac{1}{2}n(n-1)(\overline{\mathcal{R}}-b)\right\}(nH)^{-1}\nonumber\\
&=&\!\!\!\frac{1}{2}\left\{\sum_{i\neq j}(\lambda_{j}^{n+p-q+1})^{2}+\left(\sum_{j\neq i}\lambda_{j}^{n+p-q+1}\right)^{2}+n(n-1)(\overline{\mathcal{R}}-b)\right\}(nH)^{-1}.\nonumber
\end{eqnarray}
Therefore, taking into account our assumption $b<\overline{\mathcal{R}}$, we conclude that $nH-\lambda_{i}^{n+p-q+1}+\frac{n-1}{2}a>0$, which implies that $\mathcal{L}$ is an elliptic operator.

\end{proof}

To conclude this section, we record the following two algebraic lemmas, whose proofs can be found in \cite{Okumura:74} and \cite{Zhang:05}, respectively.

\begin{lemma}\label{lemma 3}
Let $\mu_i$ ($1\leq i\leq n$) be real numbers such that $\displaystyle\sum_i\mu_i=0$ and $\displaystyle\sum_i\mu_i^2=\beta$, where $\beta$ is a nonnegative constant. Then,
\begin{equation*}
-\frac{n-2}{\sqrt{n(n-1)}}\beta^3\leq\sum_i\mu_i^3\leq\frac{n-2}{\sqrt{n(n-1)}}\beta^3.
\end{equation*}
Moreover, the equality holds if and only if at least $(n-1)$ of the $\mu_i$ are equal.
\end{lemma}

\begin{lemma}\label{lemma 4}
Let $a_1,\cdots,a_n,b_1,\cdots,b_n$ be real numbers satisfying $\displaystyle\sum_ib_i=0$. Then,
\begin{equation*}
\sum_{i,j}a_ia_j(b_i-b_j)^2\leq\frac{n}{\sqrt{n-1}}\sum_ia_i^2\sum_jb_j^2.
\end{equation*}
\end{lemma}

\section{Caracterizations Results of LW-Submanifolds in $\mathbb{L}^{n+p}_q$}\label{sec:main result}

In order to prove our first characterization result, it will be essential to prove the following lower boundedness for the Cheng-Yau's modified operator $\mathcal{L}$ defined in \eqref{operadorL} acting on $nH$.

\begin{proposition}
Let $X:M^n\looparrowright \mathbb{L}^{n+p}_{q}$ be a complete spacelike LW-submanifold with parallel normalized mean curvature vector field and flat normal bundle satisfying conditions \eqref{c1}-\eqref{c3}. Suppose that the second fundamental form of $M^n$ is locally timelike and there exists an orthogonal basis for $TM$ that diagonalizes simultaneously all $B_{\varepsilon},
\varepsilon\in TM^{\perp}.$ If $c=\dfrac{c_1}{n}+2c_2$ and
$(n-1)a^2+4n(\overline{\mathcal{R}}-b)\geq 0$,
then
\begin{equation}\label{hip_1_prop}
\mathcal{L}(nH)\geq |\Phi|^2\cdot P_{n,p,c,H}(|\Phi|),
\end{equation}
where 
\begin{equation}\label{hip_2_prop}
P_{n,p,c,H}(x)=-p^2x^2-\dfrac{n(n-2)}{\sqrt{n(n-1)}}Hx+n(c-H^2).
\end{equation}
\end{proposition}

\begin{proof}

Let us consider $\{e_1,e_2,\cdots,e_{n}\}$ a local orthonormal frame on $M^{n}$ such that $h_{ij}^{\alpha}=\lambda_i^{\alpha}\delta_{ij}$, for all $\alpha\in \{n+1,n+2,\cdots,n+p\}.$ From of the Lemma \ref{formula simons} we get
\begin{eqnarray*}
2\left(\sum_{i,j,k,m,\alpha}h_{ij}^{\alpha}h_{km}^{\alpha}\overline{R}_{mijk}+\sum_{i,j,k,m,\alpha}h_{ij}^{\alpha}h_{jm}^{\alpha}\overline{R}_{mkik}\right)
&=&2\sum_{i,k,\alpha}\left((\lambda_{i}^{\alpha})^2+\lambda_i^{\alpha}\lambda_k^{\alpha}\overline{R}_{kiik}\right)\\
&=&\sum_{i,k,\alpha}\overline{R}_{ikik}\left(\lambda_i^{\alpha}-\lambda_k^{\alpha}\right)^2.
\end{eqnarray*}
Since that $L_{q}^{n+p}$ satisfies the condition $\overline{K}(u,v)\geq c_2$ for any $u,v\in TM,$ we have
\begin{eqnarray*}
2\left(\sum_{i,j,k,m,\alpha}h_{ij}^{\alpha}h_{km}^{\alpha}\overline{R}_{mijk}+\sum_{i,j,k,m,\alpha}h_{ij}^{\alpha}h_{jm}^{\alpha}\overline{R}_{mkik}\right)\geq c_2\sum_{i,k,\alpha}(\lambda_i^{\alpha}-\lambda_k^{\alpha})^2=2nc_2|\Phi|^2.
\end{eqnarray*}
Now, for each $\alpha$, consider $h^{\alpha}$ the symmetric matrix $(h_{ij}^{\alpha})$, and 
\begin{eqnarray*}
S_{\alpha\beta}=\sum_{i,j}h_{ij}^{\alpha}h_{ij}^{\beta}.
\end{eqnarray*}
Then the $(p\times p)$-matrix $(S_{\alpha\beta})$ is symmetric and we can see that is diagonalizable for a choose of $e_{n+1},e_{n+2},\cdots,e_{n+p}.$ Thence 
\begin{eqnarray*}
S_{\alpha}=S_{\alpha\alpha}=\sum_{i,j}h_{ij}^{\alpha}h_{ij}^{\alpha},
\end{eqnarray*}
and we have that 
\begin{eqnarray*}
S=\sum_{\alpha}S_{\alpha}.
\end{eqnarray*}
Since that $\mathbb{L}_{q}^{n+p}$ satisfies the condition $\overline{K}(u,\eta)=\dfrac{c_1}{n},$ for any $u\in TM$ and $\eta\in TM^{\perp},$ we obtain
\begin{eqnarray*}
\sum_{i,j,k,\alpha,\alpha'}h_{ij}^{\alpha}h_{jk}^{\alpha'}\overline{R}_{\alpha i\alpha' k}&+&\sum_{i,j,k,\alpha,\alpha'}h_{ij}^{\alpha}h_{kk}^{\alpha'}\overline{R}_{\alpha ij \alpha'}+\sum_{i,j,k,\alpha,\alpha'}h_{ij}^{\alpha}h_{ij}^{\alpha'}\overline{R}_{\alpha k\alpha'k}+\sum_{i,j,k,\alpha,\alpha'}h_{ij}^{\alpha}h_{kj}^{\alpha'}\overline{R}_{\alpha ki \alpha'}\\
&=&\sum_{i,k,\alpha}(\lambda_i^{\alpha})^2\overline{R}_{\alpha k\alpha k}-nH^2c_1
=c_1|\Phi|^2.
\end{eqnarray*}

Finally note that
$$
\dfrac{1}{2}\Delta S\geq |\nabla B|^2+2nc_2|\Phi|^2+c_1|\Phi|^2+n\sum_{i,j}h_{ij}^{n+p-q+1}H_{ij}-nH\sum_{i,j,m,\alpha}h_{ij}^{\alpha}h_{mi}^{\alpha}h_{mj}^{n+p-q+1}
$$
$$
+\sum_{\alpha,\gamma}[tr(h^{\alpha}h^{\gamma})]^2-\sum_{\alpha,\beta}[tr(h^{\alpha}h^{\beta})]^2,
$$
Therefore,
\begin{equation}\label{deltaS-inequalite}
\dfrac{1}{2}\Delta S\geq |\nabla B|^2+cn|\Phi|^2+n\sum_{i,j}h_{ij}^{n+p-q+1}H_{ij}-nH\sum_{i,j,m,\alpha}h_{ij}^{\alpha}h_{mi}^{\alpha}h_{mj}^{n+p-q+1}
\end{equation}
$$
+\sum_{\alpha,\gamma}[tr(h^{\alpha}h^{\gamma})]^2-\sum_{\alpha,\beta}[tr(h^{\alpha}h^{\beta})]^2,
$$
Thence, from of $\mathcal{L}=\square+\dfrac{n-1}{2}a\Delta,$ we have
\begin{eqnarray*}
\mathcal{L}(nH)&=&\square (nH)+\dfrac{n-1}{2}a\Delta  (nH)\\
&=&\sum_{i,j}(nH\delta_{ij}-h_{ij}^{n+p-q+1})(nH)_{ij}+\dfrac{n-1}{2}a\Delta(nH)\\
&=&n^2H\sum_{i}H_{ii}-n\sum_{i,j}h_{ij}^{n+p-q+1}H_{ij}+\dfrac{n-1}{2}a\Delta (nH)\\
&=&n^2H\Delta H-n\sum_{i,j}h_{ij}^{n+p-q+1}H_{ij}+\dfrac{n-1}{2}a\Delta (nH)
\end{eqnarray*}

Note that
\begin{eqnarray*}
\Delta H^2=2H\Delta H+2|\nabla H|^2.
\end{eqnarray*}
Thus,
\begin{eqnarray*}
\mathcal{L}(nH)=\dfrac{1}{2}\Delta (n^2H^2)-n^2|\nabla H|^2-n\sum_{i,j}h_{ij}^{n+p-q+1}H_{ij}+\dfrac{n-1}{2}a\Delta(nH).
\end{eqnarray*}
Since that $R=aH+b$ and $\mathbb{L}_{q}^{n+p}$ satisfies the conditions \eqref{c1}-\eqref{c3}, we have that $\mathcal{R}$ is constant, then from \eqref{equacaoRS} we get 
\begin{eqnarray*}
\dfrac{1}{2}n(n-1)\Delta(aH)+\dfrac{1}{2}\Delta(n^2H^2)=\dfrac{1}{2}\Delta S.
\end{eqnarray*}
Therefore, using the inequality \eqref{deltaS-inequalite} and Lemma \eqref{desigualdade_BeH} we conclude that 
\begin{eqnarray}\label{retorno}
\mathcal{L}(nH)&=&\dfrac{1}{2}\Delta S-n^2|\nabla H|^2-n\sum_{i,j}h_{ij}^{n+p-q+1}H_{ij} \nonumber \\
&\geq& |\nabla B|^2-n^2|\nabla H|^2+cn|\Phi|^2-nH\sum_{i,j,m,\alpha}h_{ij}^{\alpha}h_{mi}^{\alpha}h_{mj}^{n+p-q+1}+\sum_{\alpha,\gamma}[tr(h^{\alpha}h^{\gamma})]^2-\sum_{\alpha,\beta}[tr(h^{\alpha}h^{\beta})]^2\\
&\geq& cn|\Phi|^2-nH\sum_{i,j,m,\alpha}h_{ij}^{\alpha}h_{mi}^{\alpha}h_{mj}^{n+p-q+1}+\sum_{\alpha,\gamma}[tr(h^{\alpha}h^{\gamma})]^2-\sum_{\alpha,\beta}[tr(h^{\alpha}h^{\beta})]^2. \nonumber
\end{eqnarray}
Next, we proceed with the explicit computation of the expression involving the mixed trace terms of the second fundamental forms:
\begin{equation}\label{calculo_3_termos}
-nH\sum_{\alpha}tr[h^{n+p-q+1}(h^{\alpha})^2]+\sum_{\alpha,\gamma}[tr(h^{\alpha}h^{\gamma})]^2-\sum_{\alpha,\beta}[tr(h^{\alpha}h^{\beta})]^2.
\end{equation}
Note that
\begin{eqnarray*}
nH\sum_{i,\alpha}h_{ii}^{\alpha}h_{ii}^{\alpha}h_{ii}^{n+p-q+1}&=&nH\sum_{i,\alpha}\left(\Phi_{ii}^{\alpha}+H\delta_{\alpha (n+p-q+1)}\right)^2\left(\Phi_{ii}^{n+p-q+1}+H\right)\\
&=&nH\sum_{i,\alpha}\left((\Phi_{ii}^{\alpha})^2+2H\Phi_{ii}^{\alpha}\delta_{\alpha (n+p-q+1)}+H^2\delta_{\alpha (n+p-q+1)}\right)(\Phi_{ii}^{n+p-q+1}+H)\\
&=&nH\sum_{i,\alpha}((\Phi_{ii}^{\alpha})^2\Phi_{ii}^{n+p-q+1}+H(\Phi_{ii}^{\alpha})^2+2H\Phi_{ii}^{\alpha}\Phi_{ii}^{n+p-q+1}\delta_{\alpha (n+p-q+1)}+\\
&+&2H^2\Phi_{ii}^{\alpha}\delta_{\alpha (n+p-q+1)}+H^2\Phi_{ii}^{n+p-q+1}\delta_{\alpha (n+p-q+1)}+H^3\delta_{\alpha (n+p-q+1)})\\
&=&nH\sum_{\alpha}tr[\Phi^{n+p-q+1}\cdot (\Phi^{\alpha})^2]+nH^2|\Phi|^2+n^2H^4.
\end{eqnarray*}
Moreover,

\begin{eqnarray*}
\sum_{\alpha,\beta}[tr(h^{\alpha}h^{\beta})]^2
&=&\sum_{\alpha,\beta}[\sum_{i}h_{ii}^{\alpha}h_{ii}^{\beta}]^2 
= \sum_{\alpha,\beta}[\sum_{i}(\Phi_{ii}^{\alpha}+H\delta_{\alpha (n+p-q+1)})\cdot (\Phi_{ii}^{\beta} +H\delta_{\beta (n+p-q+1)}]^2\\
&=&\sum_{\alpha,\beta}[\sum_{i}(\Phi_{ii}^{\alpha}\Phi_{ii}^{\beta}+H\Phi_{ii}^{\alpha}\delta_{\beta (n+p-q+1)}+H\Phi_{ii}^{\beta}\delta_{\alpha (n+p-q+1)}+H^2\delta_{\alpha (n+p-q+1)}\delta_{\beta (n+p-q+1)})]^2\\
&=&\sum_{\alpha,\beta}[\sum_{i}\Phi_{ii}^{\alpha}\Phi_{ii}^{\beta}+H\delta_{\beta (n+p-q+1)}\sum_{i}\Phi_{ii}^{\alpha}+H\delta_{\alpha (n+p-q+1)}\sum_{i}\Phi_{ii}^{\beta} \\
&&+H^2\delta_{\alpha (n+p-q+1)}\delta_{\beta (n+p-q+1)}\sum_{i}1]^2\\
&=&\sum_{\alpha,\beta}[tr(\Phi^{\alpha}\Phi^{\beta})+nH^2\delta_{\alpha (n+p-q+1)}\delta_{\beta (n+p-q+1)}]^2\\
&=&\sum_{\alpha,\beta}[tr(\Phi^{\alpha}\Phi^{\beta})]^2+2nH^2\sum_{\alpha,\beta}\delta_{\alpha (n+p-q+1)}\delta_{\beta (n+p-q+1)}tr(\Phi^{\alpha}\Phi^{\beta})+n^2H^4\\
&=&\sum_{\alpha,\beta}[tr(\Phi^{\alpha}\Phi^{\beta})]^2+2nH^2|\Phi^{n+p-q+1}|^2+n^2H^4,
\end{eqnarray*}
and finally
\begin{eqnarray*}
\sum_{\alpha,\gamma}[tr(h^{\alpha}h^{\gamma})]^2&=&\sum_{\alpha,\gamma}[\sum_{i}(h_{ii}^{\alpha}h_{ii}^{\gamma})]^2\\
&=&\sum_{\alpha,\gamma}[\sum_{i}(\Phi_{ii}^{\alpha}+H\delta_{\alpha (n+p-q+1)})\Phi_{ii}^{\gamma}]^2\\
&=&\sum_{\alpha,\gamma}[\sum_{i}\Phi_{ii}^{\alpha}\Phi_{ii}^{\gamma}+H\Phi_{ii}^{\gamma}\delta_{\alpha (n+p-q+1)}]^2\\
&=&\sum_{\alpha,\gamma}[tr(\Phi^{\alpha}\Phi^{\gamma})+H\delta_{\alpha (n+p-q+1)}tr(\Phi^{\gamma})]^2\\
&=&\sum_{\alpha,\gamma}[tr(\Phi^{\alpha}\Phi^{\gamma})]^2.
\end{eqnarray*}
Substituting these three equalities into \eqref{calculo_3_termos}, we obtain
\begin{align}\label{termogrande}
-nH\sum_{i,j,m,\alpha}h_{ij}^{\alpha}&h_{mi}^{\alpha}h_{mj}^{n+p-q+1}
-\sum_{\alpha,\beta}[\operatorname{tr}(h^{\alpha}h^{\beta})]^2
+\sum_{\alpha,\gamma}[\operatorname{tr}(h^{\alpha}h^{\gamma})]^2 \nonumber \\
&= -nH\sum_{\alpha}\operatorname{tr}[\Phi^{n+p-q+1}(\Phi^{\alpha})^2]
     -nH^2|\Phi|^2
     -\sum_{\alpha,\beta}[\operatorname{tr}(\Phi^{\alpha}\Phi^{\beta})]^2
     +\sum_{\alpha,\gamma}[\operatorname{tr}(\Phi^{\alpha}\Phi^{\gamma})]^2 \nonumber \\
&\quad -2nH^2|\Phi^{n+p-q+1}|^2 \nonumber \\[4pt]
&\geq -\dfrac{n(n-2)}{\sqrt{n(n-1)}}H|\Phi|^3
      +nH^2|\Phi|^2
      -2nH^2|\Phi|^2
      -\sum_{\alpha,\beta}[\operatorname{tr}(\Phi^{\alpha}\Phi^{\beta})]^2
      +\sum_{\alpha,\gamma}[\operatorname{tr}(\Phi^{\alpha}\Phi^{\gamma})]^2 \\[4pt]
&\geq -\dfrac{n(n-2)}{\sqrt{n(n-1)}}H|\Phi|^3
      -nH^2|\Phi|^2
      -\sum_{\alpha,\beta}[\operatorname{tr}(\Phi^{\alpha}\Phi^{\beta})]^2
      +\dfrac{1}{p}|\Phi|^4, \nonumber
\end{align}
furthermore, since
\begin{eqnarray*}
[tr(\Phi^{\alpha}\Phi^{\beta})]^2&=&(\Phi_{11}^{\alpha}\Phi_{11}^{\beta}+\Phi_{22}^{\alpha}\Phi_{22}^{\beta}+\cdots+\Phi_{nn}^{\alpha}\Phi_{nn}^{\beta})^2\\
&\leq&|\Phi^{\alpha}|^2\cdot |\Phi^{\beta}|^2\leq |\Phi|^2\cdot |\Phi|^2=|\Phi|^4,
\end{eqnarray*}
hence,
\begin{eqnarray*}
\sum_{\alpha,\beta}[tr(\Phi^{\alpha}\Phi^{\beta})]^2\leq |\Phi|^4\sum_{\alpha,\alpha'}1=p^2|\Phi|^4,
\end{eqnarray*}
or yet, 
\begin{eqnarray*}
-\sum_{\alpha,\beta}[tr(\Phi^{\alpha}\Phi^{\beta})]^2\geq -p^2|\Phi|^4.
\end{eqnarray*}
Therefore, from \eqref{termogrande}
\begin{eqnarray*}
-\dfrac{n(n-2)}{\sqrt{n(n-1)}}H|\Phi|^3-nH^2|\Phi|^2-\sum_{\alpha,\beta}[tr(\Phi^{\alpha}\Phi^{\beta})]^2\geq -\dfrac{n(n-2)}{\sqrt{n(n-1)}}H|\Phi|^3-nH^2|\Phi|^2-p^2|\Phi|^4+\dfrac{1}{p}|\Phi|^4.
\end{eqnarray*}
Consequently,
\begin{equation}\label{4.7}
L(nH)\geq cn|\Phi|^2-\dfrac{n(n-2)}{\sqrt{n(n-1)}}H|\Phi|^3-nH^2|\Phi|^2-p^2|\Phi|^4+\dfrac{1}{p}|\Phi|^4
\geq |\Phi|^2P_{n,p,c,H}(|\Phi|),
\end{equation}
where
\begin{equation}
P_{n,p,c,H}(x)=-p^2x^2-\dfrac{n(n-2)}{\sqrt{n(n-1)}}Hx-n(H^2-c).
\end{equation}

\end{proof}

\begin{remark}

A crucial observation, which will be employed later, is that by returning to \eqref{retorno} and using the estimate in \eqref{4.7}, we obtain the following inequality:
\begin{equation}\label{estimativa_chave}
\mathcal{L}(nH) \geq |\nabla B|^2 - n^2 |\nabla H|^2 + |\Phi|^2  P_{n,p,c,H}(|\Phi|).
\end{equation}

\end{remark}

\begin{remark}
 \noindent
Note that 
$$
x_v=-\dfrac{1}{2p^2}\cdot \dfrac{n(n-2)H}{\sqrt{n(n-1)}}\leq 0.
$$
Due to the behavior of the polynomial \(P_{n,p,c,H}(x)\), the fact that \(x_v < 0\) together with the existence of a positive real root given by 
\begin{eqnarray*}
C(n,p,c,H)
&=&\dfrac{\sqrt{n}}{2p^2\sqrt{n-1}}\left[-(n-2)H+\sqrt{(n-2)^2H^2+4p^2(n-1)(c-H^2)}\right]>0,
\end{eqnarray*}
necessarily implies the condition \(-n(H^2-c)>0\). This turns out \(H^2 < c\) to be a natural hypothesis to assume in this context, in this case we have that.
   
\end{remark}

\subsection{Via Omori-Yau maximum principle.}

To prove our first main result, we require a criterion ensuring the existence of an Omori-type sequence associated with the $\mathcal{L}$-operator. For this purpose, we establish the following result, which extends Lemma 4 in Gomes et al. \cite{nazareno} and Proposition 3.2 in Aquino et al. \cite{aquino}.

\begin{lemma}  Let $X:M^n\looparrowright \mathbb{L}^{n+p}_{q}$ be a complete spacelike LW-submanifold satisfying conditions ~\eqref{c1},~\eqref{c2} and ~\eqref{c3} such that $R=aH+b$, with $a\geq 0$ and $(n-1)a^2+4n(\overline{\mathcal{R}}-b)\geq 0.$ If $H$ is bounded on $M^n$, $S_1\leq \sum_{i,j,\gamma\neq n+p-q+1}(h_{ij}^{\gamma})^2$ and $S_1$ is bounded on $M^n$, then there is a sequence of points $\{q_k\}_{k\in \mathbb{N}}\subset M^n$ such that 
\begin{eqnarray*}
\lim_{k}(nH)(q_k)=\sup_{M}(nH),\quad \lim_{k}|\nabla nH(q_k)|=0\quad\mbox{and}\quad \limsup_{k}\mathcal{L}(nH(q_k))\leq 0.
\end{eqnarray*}
\end{lemma}

\begin{proof}
 Let us choose a local orthonormal frame $\{e_1,e_2,\cdots,e_n\}$ on $M^n$ such that $h_{ij}^{n+p-q+1}=\lambda_{i}^{n+p-q+1}\delta_{ij}.$ From of $\mathcal{L}=\square +\dfrac{n-1}{2}a\Delta$,  we have that 
\begin{equation}
\mathcal{L}(nH)=n\sum_{i}\left(nH+\dfrac{n-1}{2}a-\lambda_{i}^{n+p-q+1}\right)H_{ii}.
\end{equation}

Thus, for all $i=1,\cdots,n$ and since that $\mathbb{L}_{q}^{n+p}$ satisfies the conditions $c_1$ and $c_3$ then from $S_2-S_1=n(n-1)(aH+b-\overline{\mathcal{R}})+n^2H^2.$ Moreover, suppose that $S_1\leq \displaystyle\sum_{i,j,\gamma\neq n+p-q+1}(h_{ij}^{\gamma})^2$
 and with straightforward computation we get

 \begin{eqnarray*}
 (\lambda_i^{n+p-q+1})^2\leq S_2-S_1&=&n(n-1)(aH+b-\overline{\mathcal{R}})+n^2H^2\\
 &=&\left(nH+\dfrac{n-1}{2}a\right)^2-\dfrac{n-1}{4}\left[(n-1)a^2+4n(\overline{\mathcal{R}}-b)\right]\\
 &\leq& \left(nH+\dfrac{n-1}{2}a\right)^2,
 \end{eqnarray*}
where we have used our assumption that $(n-1)a^2+4n(\overline{\mathcal{R}}-b)\geq 0$  to obtain
the last inequality. Consequently, for all $i=1,\cdots,n$, we have

\begin{eqnarray}\label{desigualdade de lambda}
|\lambda_i^{n+p-q+1}|\leq \left|nH+\dfrac{n-1}{2}a\right|.
\end{eqnarray}
Nonetheless, from 
\begin{eqnarray*}
R_{ijij}&=&\overline{R}_{ijij}+\sum_{i,j,\beta}(h_{ii}^{\beta}h_{jj}^{\beta}-(h_{ij}^{\beta})^2)-\sum_{i,j,\gamma}(h_{ii}^{\gamma}h_{jj}^{\gamma}-(h_{ij}^{\gamma})^2)\\
&=&\overline{R}_{ijij}+\sum_{i,j,\beta}h_{ii}^{\beta}h_{jj}^{\beta}-\sum_{i,j,\beta}(h_{ij}^{\beta})^2-\sum_{i,j,\gamma}h_{ii}^{\gamma}h_{jj}^{\gamma}+\sum_{i,j,\gamma}(h_{ij}^{\gamma})^2.
\end{eqnarray*}
Since that $H^{\beta}=0,$  for each $\beta$ then 

$$\sum_{i,j,\beta}h_{ii}^{\beta}h_{jj}^{\beta}=\sum_{i,\beta}h_{ii}^{\beta}\sum_{j}h_{jj}^{\beta}=\sum_{i,\beta}nH^{\beta}=0.$$
Hence, we get

\begin{eqnarray*}
R_{ijij}&=&\overline{R}_{ijij}-\sum_{i,j,\beta}(h_{ij}^{\beta})^2-\sum_{i,j,\gamma}h_{ii}^{\gamma}h_{jj}^{\gamma}+\sum_{i,j,\gamma}(h_{ij}^{\gamma})^2\\
&=&\overline{R}_{ijij}+S_2-S_1-\sum_{i,j,\gamma}h_{ii}^{\gamma}h_{jj}^{\gamma}\\
&\geq&\overline{R}_{ijij}+\sum_{i,j}(h_{ij}^{n+p-q+1})^2-\sum_{i,j,\gamma}h_{ii}^{\gamma}h_{jj}^{\gamma}\\
&\geq& \overline{R}_{ijij}-\sum_{i,j,\gamma}h_{ii}^{\gamma}h_{jj}^{\gamma}.
\end{eqnarray*}
Since that $S_2\leq \left(nH+\dfrac{n-1}{2}a\right)^2+S_1$ then we get that for each $i,j,\gamma$, from ~\eqref{desigualdade de lambda} we have

\begin{eqnarray*}
h_{ii}^{\gamma}h_{jj}^{\gamma}\leq |h_{ii}^{\gamma}|\cdot |h_{jj}^{\gamma}|\leq \left(nH+\dfrac{n-1}{2}a\right)^2+S_1.
\end{eqnarray*}
Therefore, since we are supposing that $H$ and $S_1$ is bounded on $M^n$ and $\mathbb{L}_{q}^{n+p}$ satisfies
the condition $c_2$, this is, $R_{ijij}\geq c_2$, it follows that the sectional curvatures of
$M^n$ are bounded from below.
Thus, we may apply the well known generalized
maximum principle of Omori \cite{omori} to the function $nH$, obtaining a sequence
of points $\{q_k\}_{k\in \mathbb{N}}$ in $M^n$
such that
\begin{eqnarray}\label{desigualdade omori yau}
\lim_{k}(nH)(q_k)=\sup(nH),\quad \lim_{k}|\nabla (nH)|=0,\quad and\quad \limsup_{k}\sum_{i}(nH_{ii})(q_k)\leq 0.
\end{eqnarray}

Since $\sup_{M}H>0$, taking subsequences if necessary, we can arrive to a sequence $\{q_k\}_{k\in \mathbb{N}}$ in $M^n$ which satisfies ~\eqref{desigualdade omori yau} and such that $H(q_k)\geq 0.$ Hence, since $a\geq 0$, we have

\begin{eqnarray*}
0&\leq& nH(q_k)+\dfrac{n-1}{2}a-|\lambda_{i}^{n+p-q+1}(q_k)|\leq nH(q_k)+\dfrac{n-1}{2}a-\lambda_{i}^{n+p-q+1}(q_k)\\
&\leq& nH(q_k)+\dfrac{n-1}{2}a+|\lambda_{i}^{n+p-q+1}(q_k)|\leq 2nH(q_k)+(n-1)a.
\end{eqnarray*}
This previous estimate shows that function $nH(q_k)+\dfrac{n-1}{2}a-\lambda_{i}^{n+p-q+1}(q_k)$
is nonnegative and bounded on $M^n$, for all $k\in \mathbb{N}.$
Therefore, taking into account ~\eqref{desigualdade omori yau},  we obtain

\begin{eqnarray*}
\limsup_{k}(\mathcal{L}(nH)(q_k))\leq n\sum_{i}\limsup_{k}\left[\left(nH+\dfrac{n-1}{2}a-\lambda_{i}^{n+p-q+1}\right)(q_k)H_{ii}(q_k)\right]\leq 0.
\end{eqnarray*}

\end{proof}

Now, we are in position to present our first theorem.

\begin{theorem}\label{teo1}
Let $X:M^n\looparrowright \mathbb{L}^{n+p}_{q}$ be a complete spacelike LW-submanifold $(a\geq0)$ with parallel normalized mean curvature vector field and flat normal bundle satisfying conditions \eqref{c1}-\eqref{c3}. Suppose that the second fundamental form of $M^n$ is locally timelike and there exists an orthogonal basis for $TM$ that diagonalizes simultaneously all $B_{\varepsilon},
\varepsilon\in TM^{\perp}.$ Assume in addition that $H^2<c$ and
$|\Phi|\leq C(n,p,c,supH)$, we conclude that: 
\begin{enumerate}
    \item[(i)] If $(n-1)a^2+4n(\overline{\mathcal{R}}-b)\geq 0$ then $M^n$ is a totally umbilical submanifold or
    \item[(ii)] If $\sup S$ is attained at some point in $M^n$ and $b<\overline{\mathcal{R}}$, then $M^n$ is an isoparametric submanifold of $\mathbb{L}_{q}^{n+p}.$
\end{enumerate} 
\end{theorem}

\begin{proof}
Since we are assuming that $a\geq 0$ and that inequality $(n-1)a^2+4n(\overline{\mathcal{R}}-b)\geq 0$,
we can apply Lemma \ref{desigualdade_BeH} to the function $nH$ in order to obtain a sequence of points $\{q_k\}_{k\in \mathbb{N}}\subset M^n$ such that
\begin{equation}\label{4.8}
\limsup_{k}nH(q_k)=\sup_{M}(nH),\quad\mbox{and}\quad \limsup_{k}\mathcal{L}(nH)(q_k)\leq 0.
\end{equation}
Thus, from \eqref{4.7} and \eqref{4.8} we have 
\begin{equation}\label{4.9}
0\geq \limsup_{k}\mathcal{L}(nH)(q_k)\geq \sup_{M}|\Phi|^2P_{\sup H,n,p,c}\left(\sup_{M}|\Phi|\right).
\end{equation}
On the other hand, our hypothesis imposed on $|\Phi|$ guarantees us that $|\Phi|\equiv 0$ or $\sup_{M}|\Phi|>0.$ If $|\Phi|\equiv 0$ then $M^n$ is umbilical totally. Suppose that $\sup_{M}|\Phi|>0.$ Therefore, from \eqref{4.9} we conclude that 
\begin{equation}\label{4.10}
P_{n,p,c,\sup H}\left(\sup_{M}|\Phi|\right)\leq 0.
\end{equation}
From our restrictions on $H$ and $|\Phi|$, we have that $P_{H,n,p,c}(|\Phi|)\geq  0$, with $P_{H,n,p,c}(|\Phi|)=0$ if, and only if, $|\Phi|=C(n,p,c,H).$ Consequently, from \eqref{4.10}
\begin{equation}\label{hipparaparabolicidade}
\sup_{M}|\Phi|=C(n,p,c,supH).
\end{equation}
In this case, from \eqref{4.7} and taking into account once more the behavior of the
function $|\Phi|$, we get that $\mathcal{L}(nH)\geq 0.$ But, since we are assuming that
$b<\overline{\mathcal{R}}$, item (i) of Lemma ~\ref{L eliptico} guarantees that $\mathcal{L}$ is elliptic. Consequently, since we are also supposing that the supremum of $S$ on $M^n$ is attained at some point of $M^n$, by Hopf’s Maximum Principle we conclude that $|\Phi|$ is constant on $M^n$, consequently the same holds for $H$ and thus all the inequalities previously obtained become equalities. Hence, from \eqref{estimativa_chave} we obtain
\begin{eqnarray*}
\sum_{i,j,k,\alpha}(h_{ijk}^{\alpha})^2=|\nabla B|^2=n^2|\nabla H|^2=0,
\end{eqnarray*}
that is, $h_{ijk}^{\alpha}=0$ for all $i,j\in \{1,2,\cdots,n\}.$ 
Therefore, we conclude that $M^n$ is an
isoparametric submanifold of $\mathbb{L}_{q}^{n+p}.$

\end{proof}

\begin{remark}
Observe that if $a = 0$ and the scalar curvature $R$ is constant, then the condition
$
(n-1)a^2 + 4n(\overline{\mathcal{R}} - b) \geq 0
$
is automatically satisfied.
\end{remark}

In this setting, we may state the following corollary

\begin{corollary}

Let $X:M^n\looparrowright \mathbb{L}^{n+p}_{q}$ be a complete spacelike LW-submanifold with parallel normalized mean curvature vector field and flat normal bundle satisfying conditions \eqref{c1}-\eqref{c3} and constant scalar curvature $R\leq \overline{\mathcal{R}}$. Suppose that the second fundamental form of $M^n$ is locally timelike and there exists an orthogonal basis for $TM$ that diagonalizes simultaneously all $B_{\varepsilon}, \varepsilon\in TM^{\perp}.$ Assume in addition that $H^2<c$ and $|\Phi|\leq C(n,p,c,supH)$, then either
\begin{enumerate}
    \item[(i)] $M^n$ is a totally umbilical submanifold or
    \item[(ii)] If $\sup S$ is attained at some point in $M^n$ and $b<\overline{\mathcal{R}}$, then $M^n$ is an isoparametric submanifold of $\mathbb{L}_{q}^{n+p}.$
\end{enumerate} 

\end{corollary}

\subsection{Via $\mathcal{L}$-parabolicity}

We recall that a Riemannian manifold $M^n$ is said to be \emph{parabolic} (with respect to the Laplacian operator) if the constant functions are the only subharmonic functions on $M^n$ which are bounded from above; that is for a function $u \in C^2(M)$
$$
\Delta u \geq 0 \quad \text{and} \quad u \leq u^* \ll \infty \quad \text{implies} \quad u = \text{constant}.
$$
From a physical viewpoint, parabolicity is closely related to the recurrence of the Brownian motion. Roughly speaking, the parabolicity is equivalent to the property that all particles will pass through any open set at an arbitrarily large time, for more details see \cite{grigorian:99}.

Extending this previous concept for the operator $\mathcal{L}$ defined in (2.10), $M^n$ is said to be \emph{$\mathcal{L}$-parabolic} if the constant functions are the only functions $u \in C^2(M)$ which are bounded from above and satisfies $\mathcal{L}u \geq 0$, that is, for a function $u \in C^2(M)$,
$$
\mathcal{L}u \geq 0 \quad \text{and} \quad u \leq u^* \ll \infty \quad \text{implies} \quad u = \text{constant}.
$$

In this setting, we obtain the following gap result.

\begin{theorem}
Let $X:M^n\looparrowright \mathbb{L}^{n+p}_{q}$ be a complete spacelike LW-submanifold $(a\geq0)$ with parallel normalized mean curvature vector field and flat normal bundle satisfying conditions \eqref{c1}-\eqref{c3}. Suppose that the second fundamental form of $M^n$ is locally timelike and there exists an orthogonal basis for $TM$ that diagonalizes simultaneously all $B_{\varepsilon},
\varepsilon\in TM^{\perp}.$ Assume in addition that $M^n$ is $\mathcal{L}-$parabolic, $H^2<c$ and $0\leq |\Phi|\leq C(n,p,c,supH)$, then either $|\Phi|\equiv 0$ and $M^n$ is a totally umbilical, or $|\Phi| =  C(n,p,c,supH)$ and $M^n$ is an isoparametric submanifold.
\end{theorem}

\begin{proof} Suppose by contradiction that $M^n$ is not totally umbilical, Since we are assuming that $0 \leq |\Phi|\leq C(n,p,c,supH)$, we obtain
$$
0< \sup_{M}|\Phi| \leq C(n,p,c,supH),
$$
in the sense, from \eqref{hipparaparabolicidade} of Theorem \ref{teo1} we get that
\begin{equation}
\sup_M |\Phi| = C(n,p,c,supH).
\end{equation}
Furthermore, from \eqref{4.9} jointly with the $\mathcal{L}$-parabolicity of $M^n$ we conclude that $|\Phi|$ must be constant and identically equal to $C(n,p,c,supH)$. Therefore, at this point we can proceed as in the last part of the proof of Theorem \ref{teo1} to conclude the result.
  
\end{proof}

\begin{remark}
When the ambient space $\mathbb{L}^{n+p}_q$ is supposed to be Einstein, reasoning as in the first part of the proof of \cite[Theorem 1.1]{Lima: 13}, it is not difficult to verify that
\begin{equation}\label{newL}
\mathcal{L}(f) = \mathrm{div}(\mathcal{P}(\nabla f)),
\end{equation}
where $P$ is just the operator defined by
\begin{equation}\label{newP}
\mathcal{P} = \left( nH + \dfrac{n-1}{2}a \right) I - h^{n+1}.
\end{equation}
\end{remark}

\subsection{Via integrability property}

In \cite{yau_new}, S. T. Yau established the following version of Stokes' theorem on an $n$-dimensional complete noncompact Riemannian manifold $M^n$: If $\omega \in \Omega^{n-1}(M)$ is an $(n-1)$-differential form on $M^n$, then there exists a sequence $\{B_i\}$ of domains on $M^n$ such that
$$
B_i \subset B_{i+1}, \quad M^n = \bigcup_{i \geq 1} B_i \quad \text{and} \quad \lim_{i \to \infty} \int_{B_i} d\omega = 0.
$$

Supposing that $M^n$ is oriented by the volume element $dM$ and considering the contraction of $dM$ in the direction of a smooth vector field $X$ on $M^n$, that is,
$$
d\omega = \iota_X dM,
$$
A. Caminha in \cite{caminha:11} obtained a suitable consequence of Yau's result, which is described below (specifically, see \cite[Proposition 2.1]{caminha:11}). In what follows, $L^1(M)$ stands for the space of Lebesgue integrable functions on $M^n$.

\begin{lemma}\label{lemmaL1}
Let $X$ be a smooth vector field on the $n$-dimensional complete oriented Riemannian manifold $M^n$, such that $\mathrm{div}\, X$ does not change sign on $M^n$. If $|X| \in L^1(M)$, then $\mathrm{div}\, X = 0$.
\end{lemma}

We close our paper applying Lemma 2 in order to obtain the following characterization result.

\begin{theorem}
Let $X:M^n\looparrowright \mathbb{L}^{n+p}_{q}$ be a complete spacelike LW-submanifold $(a\geq0)$ with parallel normalized mean curvature vector field and flat normal bundle satisfying conditions \eqref{c1}-\eqref{c3}. Suppose that the second fundamental form of $M^n$ is locally timelike and there exists an orthogonal basis for $TM$ that diagonalizes simultaneously all $B_{\varepsilon},
\varepsilon\in TM^{\perp}.$ If $|\nabla H| \in L^1(M)$, $H^2<c$ and 
$|\Phi|\leq C(n,p,c,supH)$, then either $M^n$ is a totally umbilical, or $|\Phi| =  C(n,p,c,supH)$ and $M^n$ is an isoparametric submanifold.
\end{theorem}

\begin{proof}

Since $R = aH + b$ and the mean curvature function $H$ is bounded on $M^n$, from \eqref{equacaoRS}, it follows that the second fundamental form $B$ is also bounded on $M^n$. Consequently, by \eqref{newP}, the operator $\mathcal{P}$ is bounded, that is, there exists a positive constant $\beta > 0$ such that
$$
|\mathcal{P}| \leq \beta \quad \text{on } M^n.
$$
Moreover, since we assume that $|\nabla H| \in L^1(M)$, it follows from \eqref{newP} that
\begin{equation}
|\mathcal{P}(\nabla H)| \leq |\mathcal{P}|\,|\nabla H|
\leq \beta |\nabla H| \in L^1(M).
\end{equation}
Therefore, taking into account \eqref{newL} and \eqref{newP}, and applying Lemma~\ref{lemmaL1}, we obtain
\begin{equation}\label{div0}
\mathcal{L}(nH) = \mathrm{div}(\mathcal{P}(nH)) = 0 \quad \text{on } M^n.
\end{equation}
On the other hand, using the estimate
$$
0 \leq |\Phi| \leq C(n,p,c,\sup H),
$$
together with \eqref{estimativa_chave} and \eqref{div0}, we infer that
\begin{equation}
0 = \mathcal{L}(nH)
\geq |\nabla B|^2 - n^2 |\nabla H|^2
+ |\Phi|^2 P_{n,p,c,H}(|\Phi|)
\geq 0.
\end{equation}
Hence, all inequalities above must be equalities. In particular, we conclude that
$
|\nabla B|^2 = n^2 |\nabla H|^2.$
Consequently, Lemma~\ref{desigualdade_BeH} implies that the mean curvature $H$ is constant on $M^n$. As a consequence, we obtain
$$
\sum_{i,j,k,\alpha} (h_{ijk}^{\alpha})^2
= |\nabla B|^2
= n^2 |\nabla H|^2
= 0,
$$
which shows that $h_{ijk}^{\alpha} = 0$ for all indices. Therefore, $M^n$ is an isoparametric submanifold. The proof is completed by the same argument used in the final part of the proof of Theorem~\ref{teo1}.

\end{proof}

\section*{Acknowledgements}
The second author acknowledges the financial support of the Brazilian National Council for Scientific and Technological Development (CNPq) through the Universal Project, grant number 406078/2025-4.

\bibliographystyle{amsplain}

\end{document}